\documentclass[12pt]{amsart}
\usepackage{amsmath, amssymb, amsthm, latexsym}
\usepackage{amssymb}
\usepackage{latexsym}
\usepackage[center]{caption}
\usepackage{tikz}
\newcounter{braid}
\newcounter{strands}
\pgfkeyssetvalue{/tikz/braid height}{1cm}
\pgfkeyssetvalue{/tikz/braid width}{1cm}
\pgfkeyssetvalue{/tikz/braid start}{(0,0)}
\pgfkeyssetvalue{/tikz/braid colour}{black}
\pgfkeys{/tikz/strands/.code={\setcounter{strands}{#1}}}

\makeatletter
\def\cross{%
  \@ifnextchar^{\message{Got sup}\cross@sup}{\cross@sub}}

\def\cross@sup^#1_#2{\render@cross{#2}{#1}}

\def\cross@sub_#1{\@ifnextchar^{\cross@@sub{#1}}{\render@cross{#1}{1}}}

\def\cross@@sub#1^#2{\render@cross{#1}{#2}}

\def\render@cross#1#2{
  \def\strand{#1}
  \def\crossing{#2}
  \pgfmathsetmacro{\cross@y}{-\value{braid}*\braid@h}
  \pgfmathtruncatemacro{\nextstrand}{#1+1}
  \foreach \thread in {1,...,\value{strands}}
  {
    \pgfmathsetmacro{\strand@x}{\thread * \braid@w}
    \ifnum\thread=\strand
    \pgfmathsetmacro{\over@x}{\strand * \braid@w + .5*(1 - \crossing) * \braid@w}
    \pgfmathsetmacro{\under@x}{\strand * \braid@w + .5*(1 + \crossing) * \braid@w}
    \draw[braid] \pgfkeysvalueof{/tikz/braid start} +(\under@x pt,\cross@y pt) to[out=-90,in=90] +(\over@x pt,\cross@y pt -\braid@h);
    \draw[braid] \pgfkeysvalueof{/tikz/braid start} +(\over@x pt,\cross@y pt) to[out=-90,in=90] +(\under@x pt,\cross@y pt -\braid@h);
    \else
    \ifnum\thread=\nextstrand
    \else
     \draw[braid] \pgfkeysvalueof{/tikz/braid start} ++(\strand@x pt,\cross@y pt) -- ++(0,-\braid@h);
    \fi
   \fi
  }
  \stepcounter{braid}
}

\tikzset{braid/.style={double=\pgfkeysvalueof{/tikz/braid colour},double distance=1pt,line width=2pt,white}}

\newcommand{\braid}[2][]{%
  \begingroup
  \pgfkeys{/tikz/strands=2}
  \tikzset{#1}
  \pgfkeysgetvalue{/tikz/braid width}{\braid@w}
  \pgfkeysgetvalue{/tikz/braid height}{\braid@h}
  \setcounter{braid}{0}
  \let\sigma=\cross
  #2
  \endgroup
}
\makeatother

\input xypic
\newtheorem{theorem}{Theorem}
\newtheorem{proposition}[theorem]{Proposition}

\newtheorem{lemma}[theorem]{Lemma}

\newtheorem{corollary}[theorem]{Corollary}

\numberwithin{equation}{section}
\def\Proof{\medskip\noindent{\bf Proof: }}

\def\Z{\mathbb{Z}}

\def\Q{\mathbb{Q}}

\def\R{\mathbb{R}}

\def\F{\mathbb{F}}

\def\Pi{\mathbb{P}^{\infty}}

\def\md{\mathcal{D}}

\def\qed{\hfill$\square$\medskip}

\def\Zpk{\mathbb{Z}/p^{k}}
\def\Zpk1{\mathbb{Z}/p^{k-1}}

\newcommand{\rref}[1]{(\ref{#1})}

\newcommand{\fracd}[2]{\frac{\displaystyle #1}{\displaystyle #2}}

\newcommand{\beg}[2]{\begin{equation}\label{#1}#2\end{equation}}
\def\r{\rightarrow}
\def\mc{\mathcal{C}}

\def\F{\mathbb{F}}

\def\sl2{\widetilde{SL_{2}(\Z)}}
\def\smin{\smallsetminus}
\def\md{\mathcal{D}}
\def\rank{\operatorname{rank}}

\title[Spanning tree cohomology]{A spanning tree cohomology theory for links}
\author{Daniel Kriz, Igor Kriz}
\thanks{The first author was supported by the Princeton Summer Research Program. 
The second author was supported by NSF grant DMS 1102614}


\begin{document}

\maketitle

\begin{abstract}
In their recent preprint, Baldwin, Ozsv\'{a}th and Szab\'{o} defined a twisted 
version (with coefficients in a Novikov ring) of a spectral sequence, 
previously defined by Ozsv\'{a}th and Szab\'{o}, from Khovanov homology to 
Heegaard-Floer homology of the branched double cover along a link. In their preprint, 
they give a combinatorial interpretation of the 
$E_3$-term of their spectral sequence. The main purpose of the present 
paper is to prove directly that this $E_3$-term is a link invariant. 
We also give some concrete examples of computation of the invariant.
\end{abstract}

\section{Introduction}

\vspace{3mm}

The last decade or so has been a fruitful time for invention of
a new generation of knot invariants. This
includes Khovanov homology \cite{khovanov, dbar}, which is a 
sequence of homology groups whose Euler characteristic
is the Jones polynomial, and knot Floer homology of Ozsv\'{a}th and Szab\'{o}
\cite{ost, ost2, ost3, mos}, which is similarly related to the Alexander 
polynomial. In \cite{os2}, Ozsv\'{a}th and Szab\'{o} considered yet
another link invariant, namely the Heegaard-Floer homology of the 
branched double cover of $S^3$ along $L$, and discovered a spectral sequence
from Khovanov homology to $\widehat{HF}(\Sigma(L))$. Baldwin \cite{baldwin}
proved that every $E^r$-term of this spectral sequence is a link invariant.

\vspace{3mm}

In a still more recent paper \cite{os} (which is to appear soon), Baldwin, Ozsv\'{a}th and Szab\'{o} 
introduced
a variant, namely perturbed Heegaard Floer homology with coefficients in a ``Novikov
ring''. They also constructed a spectral sequence analogous to \cite{os}
in this new setting. Curiously, the behavior of this modified construction
is in a way quite distinct from \cite{os}. Instead of the $E_2$-term
being Khovanov homology, $d_1$ is, in fact, trivial, and the cochain complex
$(E_2,d_2)$ has a combinatorial description given in \cite{os}. 
In fact, basis elements of $E_2$ can be identified with Kauffman states
for the Alexander polynomial \cite{kauf}, the set of which is considerably
smaller than the basis of the chain complex calculating Khovanov homology.

\vspace{3mm}

The main purpose of this paper is to show that the $E_3$-term of the
spectral sequence mentioned in the last paragraph, which we call
BOS cohomology (BOS stands for Baldwin-Ozsv\'{a}th-Szab\'{o}), is an invariant of oriented links. 
This was conjectured by John Baldwin. It is
proved in \cite{os} that the next possible differential in this spectral
sequence is $d_6$. It is therefore natural to ask if the spectral sequence
collapses. This is not known at present. Even if the spectral sequence
collapses, the $E_3$ term is a new invariant, since it is graded, while
the spectral sequence is, at least a priori, not. 

\vspace{3mm}

We aim for the present paper to be entirely self-contained. In fact,
we use no Floer homology techniques; our methods are entirely
algebraic. We
define all the concepts we are using in Section \ref{snot} below,
and state our main result precisely.
We also prove from first principles that the Baldwin-Ozsv\'{a}th-Szab\'{o} $d_2$-differential
satisfies $dd=0$, without referring to the spectral sequence. In Section
\ref{efund}, we prove a fundamental lemma which allows us to vary
the field of coefficients. This is a key step in proving invariance under
the Reidemeister moves, which is proved in Sections \ref{sr2}, \ref{sr3}.
Ultimately, the main tool used in those proofs are algebraic identities
involving M\"{o}bius transformations over fields of characteristic $2$.

\vspace{3mm}

\noindent
{\bf Acknowledgement:} The authors are indebted to John Baldwin 
and Zolt\'{a}n Szab\'{o} for 
sharing their preprint \cite{os} with us, and for helpful discussions.

\vspace{3mm}

\section{Preliminaries, and statement of the main result}
\label{snot}

Consider an oriented link $L$ in $S^3$ with generic projection $\md$. Throughout this
paper, we will use the following
assumption:
\beg{lass}{\parbox{3.5in}{Every connected component of $S^2\smin \md$ is simply
connected.}
\tag{A}}
Following \cite{os}, we denote by $(C(\md),\Psi)$ the cochain complex which is the 
$E_2$-term of the spectral sequence \cite{os} converging to the Heegaard-Floer twisted
homology $\widehat{HF}(\Sigma(L))$ where $\Sigma(L)$ is the branched double
cover of $S^3$ along the link $L$. In particular, with $\md$, there is associated
a planar {\em black graph} $B(\md)$, and a dual planar 
{\em white graph} $W(\md)$. $C(\md)$ is the $\Lambda$-module
on the basis $K(\md)$, which is the set of all {\em Kauffman states},
which are the spanning trees of $B(\md)$. 
We color the connected components of $S^2\smin \md$ (called {\em faces})
black and white
so that a black and white face never share an edge. The vertices of $B(\md)$
consist of faces colored black, and edges go through crossings of $\md$.
The white graph is defined dually where the vertices are the faces which are colored
white. Note that two vertices of the graph $B(\md)$ may have connected by multiple
edges and loops are also possible (similarly for the graph $W(\md)$). Because of
this, technically, $G=B(\md), W(\md)$ must be defined as $1$-dimensional
CW complexes, i.e. there are sets of {\em vertices} $V(G)$ and {\em edges}
$E(G)$ and source and target maps $S,T:E(G)\r V(G)$. However, when it is clear which
edge connecting two vertices $x,y\in V(G)$ we have in mind, we will also abuse notation
to write $\{x,y\}\in E(G)$. Note that there is a canonical bijection
\beg{etau}{\tau:E(B(\md))\r E(W(\md))
}
sending a black edge to the white edge passing over the same crossing of $\md$.
We may additionally speak of {\em orientations related by $\tau$} when the
white edge orientation is obtained by rotating the black edge orientation by
$90^\circ$ degrees counter-clockwise. Note also that for a black spanning tree $T$,
there is a unique dual white spanning tree $\tau(T)$ which contains precisely
the edges $\tau(e)$ where $e\notin E(T)$. We also note that $B(\md)$ and $W(\md)$
are planar graphs; by the Assumption \rref{lass}, together with their faces, these
graphs specify ``Poincar\'{e}-dual'' CW-decompositions of $S^2$,
which will be denoted by $B(S^2)$, $W(S^2)$, respectively.

\vspace{3mm}
To each edge $e$ of $B(\md)$ there
is now assigned a {\em height} $h(e)\in\{0,1\}$ which depends on
the direction $e$ crosses the crossing of $\md$. The convention is arbitrary,
but must be fixed. Actually, more precisely, there is another
convention which must be fixed, namely {\em positive and negative
crossings}, and both conventions must be
related appropriately. Use an isotopic deformation, if necessary, to make
the arcs cross at a $90^\circ$ angle. Let us then say that a crossing is
{\em positive} when the upper arc of the crossing
is oriented in the direction 
$90^\circ$ clockwise from the orientation
of the bottom edge. In the other case, we speak
of a {\em negative} crossing (see Figure 1).
Let the number of positive
resp. negative crossings of the projection $\mathcal{D}$ be
$n_+$ resp. $n_-$. To define the height of the black edge, draw the black
graph so that an edge passes the corresponding crossing at precisely a $45^\circ$
angle.
Now the height of a black edge $e$ passing through a crossing
is $0$ if the upper
arc of the crossing is $45^\circ$ counter-clockwise from the edge $e$ and
$1$ otherwise
(this is independent of orientation; see Figure 2). 

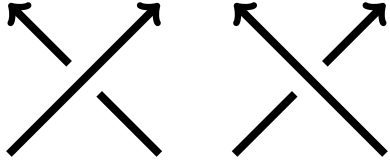
\begin{figure}
\begin{tikzpicture}[line width=3pt]
\draw[->] (-3,0) -- (-1,2);
\draw[->] (-2.2,1.2) -- (-3,2);
\draw (-1.8,.8) -- (-1,0);

\draw (0,0) -- (.8,.8);
\draw[->] (1.2,1.2) -- (2,2);
\draw[->]  (2,0) -- (0,2);
\end{tikzpicture}
\caption{A positive crossing and a negative crossing}
\end{figure}

\begin{figure}
\label{fheight}
\begin{tikzpicture}[line width=3pt]
\draw (-4,2) -- (-3.2,2.8);
\draw (-4,4) -- (-2,2);
\draw (-2.8,3.2) -- (-2,4);

\draw (1,2) -- (1.8,2.8);
\draw (2.2,3.2) -- (3,4);
\draw (1,4) -- (3,2);

\draw[line width=1.5pt] (-3,4.25) -- (-3,1.75);
\draw[line width=1.5pt] (.75,3) -- (3.25,3);

\draw (-2.95,3.75) node[anchor=east] {$e$};
\draw (1.25, 2.9) node[anchor=south] {$e$};

\filldraw[black]
 (-3,4.25) circle (2pt)
(-3,1.75) circle (2pt)
(.75,3) circle (2pt)
(3.25,3) circle (2pt);

\draw (-3,1) node{$h(e)=0$};
\draw (2,1) node{$h(e)=1$};

\end{tikzpicture}

\caption{The height of an edge through a crossing}
\end{figure}
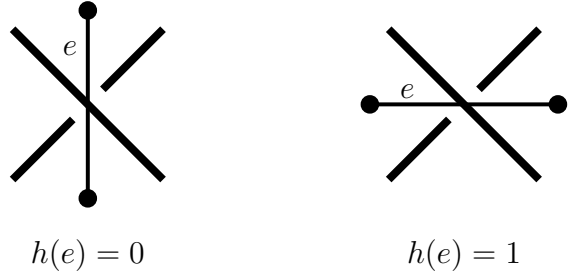

We may make different conventions regarding
heights of white edges. It is perhaps most natural to set
\beg{etauh}{h(\tau(e))=1-h(e).
}
(That way, in the definition of a chain complex below, 
if we swap faces colored white and black, we will obtain manifestly
isomorphic cochain complexes.) For a spanning tree $T$ of $B(\md)$, we set
\beg{etreeh}{h(T)=\sum_{e\in E(T)} h(e) + \sum_{e\notin E(T)} (1-h(e)).}

\vspace{3mm}
Now consider the {\em Novikov field} $\Lambda$, by which we mean
the set of
elements of the form
$$\sum_{r\in \R} a_r T^r,$$
where for each $N\in\R$ there are only finitely many $r$ with $0\neq a_r\in\Z/2$.

We construct a cochain complex whose summand in degree
$d$ is the free $\Lambda$-module
(where $\Lambda$ is a field specified below) on all spanning trees of
height $h=2d+n_-$. In other words, 
\beg{edegree}{d=\frac{1}{2}(h-n_-),}
and we notice that $\in \frac{1}{2}\Z$.
It is not difficult to see, however, that for a given projection $\mathcal{D}$,
all degrees which can occur differ by integers, or, in other words, 
heights of any two spanning trees $T,T^\prime$ differ by even numbers (this is shown
by induction on the number of edges in $E(T)\smin E(T^\prime)$.
The differential $\Psi$ additionally depends on {\em weights} which are
$\Z$-linearly independent (except as explicitly specified below) real numbers
$w(e)$ assigned to each oriented black edge $e$. Reversing orientation of
an edge has the effect of reversing the sign of $w(e)$. We set
\beg{etauw}{w(\tau(e))=w(e).}

\vspace{3mm}
To define $\Psi$, we also choose a {\em base point} which is an arc of $\md$. Then
there is precisely one adjacent black vertex and one adjacent white vertex
which are called the {\em black base point} and {\em white base point}.
Now let $T\in K(\md)$ and $T^\prime\in K(\md)$ where 
there exist black edges $e,f\in E(B(\md))$ with $h(e)=0$, $h(f)=1$,
$$E(T^\prime)=(E(T)\smin \{e\})\cup \{f\}.$$
(Note that $h(T^\prime)=h(T)+2$.)
Consider then the unique black circuit $c$ specified by 
the edges of $E(T)\cup E(T^\prime)$. We orient the circuit consistently 
(clockwise or counterclockwise) so that $f$ is
oriented from the connected component $C$ of
$T\cap T^\prime$ not containing the base point to the
connected component $C^\prime$ containing the base point.
Then let $A(T,T^\prime)$ be the sum of the weights of the edges
of the circuit $c$, oriented as specified above. We obtain another number
$B(T,T^\prime)$ as the sum of the weights of all black edges from
a vertex of $C$ to a vertex of $C^\prime$. Then define
\beg{epsi}{\Psi(T)=\sum_{T^\prime} \left(\frac{1}{1+T^{A(T,T^\prime)}}+
\frac{1}{1+T^{B(T,T^\prime)}}\right)T^\prime.
}

\vspace{3mm}
Note again that $\Psi$ raises $h$ by $2$.

\vspace{3mm}
\noindent
{\bf Comment:} It is worth mentioning that the system of weights is really a real-valued
cellular $1$-cochain on $B(\mathcal{D})$, which induces a cellular $1$-cochain on $W(\mathcal{D})$
via \rref{etauw}. (Note that, of course, these are automatically $1$-cocycles, since $B(\mathcal{D})$,
$W(\mathcal{D})$ are $1$-dimensional.)
Now the linear independence condition assures that the map $w:H_1(B(\mathcal{D}),\Z)\r
\R$ is injective. Note that this makes $A(T,T^\prime)$ and $B(T,T^\prime)$
evaluations of the cocycles $w$ on non-zero homology classes, thus 
showing in particular that the denominators of \rref{epsi} are non-zero. It is worth noting that
in the next section, we shall prove a ``fundamental lemma'' (Lemma \ref{l1}
below) which will show that the induced map $H_1(B(\mathcal{D}),\Z)\oplus H_1(W(\mathcal{D}),\Z)
\r \R$ is also injective. 

\vspace{3mm}
\begin{lemma}
\label{lpsipsi}
We have
\beg{epsipsi}{\Psi\circ\Psi=0.} 
\end{lemma}

We will prove this at the end of this section after some
re-statements. Nevertheless, it may be difficult to guess
the formula \rref{epsi} directly. Baldwin, Ozsv\'{a}th and Szab\'{o} \cite{os}
obtained the complex $(C(\md),\Psi)$ as the $E_2$-term of a spectral sequence
calculating twisted Heegaard-Floer homology of the branched double cover $\Sigma(L)$ of
$S^3$ along the link $L$, which implies \rref{epsipsi}.

\vspace{3mm}
It is worth noting that in the definition of the differential $\Psi$, black and white
do not play a symmetrical role: if we interpret $B(T,T^\prime)$ as the
sum of weights of white edges on a consistently oriented white circuit $w$, then the orientation
of $w$ does not depend on the choice of edges $e,f$, as long as they cross two edges of $w$ of
the required heights. On the other hand, the orientation of the black circuit $c$ discussed
above clearly can depend on the choice of the edges $e,f$ in it. 

\vspace{3mm}
Nevertheless, it turns out that we have the following

\begin{proposition}
\label{ppsym}
The value of $\Psi$ is symmetrical in black
and white, and is independent of the choice of base points. 
\end{proposition}

\Proof
Let us first discuss independence of the choice of base point. Clearly, the definition presented above
only depends on the choice of black base point. Now when the black base point
moves from the connected component $C$ to the component $C^\prime$, both of
the numbers $A(T,T^\prime)$, $B(T,T^\prime)$ get multiplied by $-1$. Thus, the
differential remains the same by the formula
\beg{eklsign}{\frac{1}{1+k}+\frac{1}{1+\ell}=\frac{1}{1+k^{-1}}+\frac{1}{1+\ell^{-1}},
}
which is valid in fields of characteristic $2$. Let us now turn to the question of
swapping black and white. By definition, the differential after the swap
will be equal to the original differential when $T$, $T^\prime$ are such that
the white base point is inside the black circuit $c$ if and only if $c$ is
oriented clockwise (note that the roles of $e,f$ are the opposite
from the roles of the white edges crossing them). 
By \rref{eklsign}, then, again, the differential doesn't
change when the white base point is in the other connected component of $S^2\smin c$,
and hence is equal to the original differential.
\qed

\vspace{3mm}
It is worth noting that there is one variant $\Psi^\prime$ of the definition of
$\Psi$ which does produce possibly different cohomology, namely if we change the
convention so that one of the numbers $A(T,T^\prime)$, $B(T,T^\prime)$ remains the same, and
the other is multiplied by $-1$. We see that one way of achieving this is
by swapping the roles of $e$ and $f$ in determining the orientation of $c$.
Therefore, by the universal coefficient theorem,
the cohomology of the complex modified in this way is isomorphic to the 
dual of the
$\Psi$-cohomology of the mirror projection $\mathcal{D}^\prime$ to $\mathcal{D}$
of the mirror link $L^\prime$ 
of $L$. More precisely, counting the number of positive
and negative crossings, and keeping in mind that a positive crossing turns
into negative and vice versa in the mirror projection,
the sign of the cohomological degree gets reversed. Thus, we have proved

\vspace{3mm}
\begin{proposition}
\label{ppsym1}
We have
$$H^i(C(\mathcal{D},\Psi^\prime))=H^{-i}(C(\mathcal{D}^\prime,\Psi)).$$
\end{proposition}
\qed

\vspace{3mm}
It may be tempting to
call the cohomology of $(C(\md),\Psi)$ twisted Khovanov homology, but this
is, in fact, inaccurate, since it is the $E_3$-term (and not $E_2$-term)
of the twisted analogue of the spectral sequence \cite{os2}
from $E_2=$ Khovanov homology to Heegaard Floer homology of $\Sigma(L)$.
Because of this, we use the term BOS cohomology. During the refereeing process of
this paper, it also came to our attention that the term `twisted Khovanov homology'
was being used by Roberts \cite{roberts} and Jaeger \cite{jaeg}.

\vspace{3mm}
The field $\Lambda$ and the selection of arbitrary weights with the requirement
that they be linearly independent over $\Q$ may seem unnatural. In fact, 
it can be restated. First recall that in computing the numbers $A(T,T^\prime)$,
we always sum the weights of edges of a consistently oriented circuit $c$. The
circuit determines a cellular $1$-cycle, i.e. an element 
$$\overline{c}\in Z_{1}^{cell}(B(S^2),\Z).$$
Now since $H^1(S^2,\Z)=0$, we have $\overline{c}=dx$ where 
$x\in C_{2}^{cell}(B(S^2),\Z)$. The generators of $C_{2}^{cell}(B(S^2))$ 
are faces $f$, which, by convention, we orient so that the circuit $df$ is oriented
counter-clockwise for the bounded faces and
clockwise for the unbounded face. Then the sum $\sum f$ 
of all the faces of $B(S^2)$ is a $2$-cycle representing the
fundamental class of $S^2$, and $x$ is determined uniquely up to adding integral
multiples of $\sum f$. This means that if we choose a field $F$ of characteristic $2$, 
and for each face $f$ we choose an
element
$u_f\in F$, with the relation
\beg{erel1}{\prod_f u_f=1,}
we may assign to $c$ a well defined element
$$\alpha(T,T^\prime)=\prod_f (u_f)^{\epsilon_f}$$
where 
$$x=\sum_f \epsilon_f f.$$
Similarly, $B(T,T^\prime)$ may be interpreted as the cellular $1$-cochain
in $C^{1}_{cell}(B(S^2))$ which is of the form $\delta(y)$ where the value of
$y$ is $1$ on all the vertices of $C$, and $0$ on all the vertices of $C^\prime$.
The sum $\sum v$ of all vertices of $B(S^2)$ satisfies $\delta(\sum v)=0$
(it represents the unit element in $H^0(S^2,\Z)$, so if we choose, again,
an element $z_v\in F$ for every vertex $v$, subject to the relation
\beg{erel2}{\prod_v z_v=1,}
then we may assign to $T,T^\prime$ a well defined element
$$\beta(T,T^\prime)=\prod_v (z_v)^{y(v)}.$$
Then if the variables $u_f$, $z_v$ belong to any field $F$ of characteristic $2$,
we may define $C(\md,F,(u_f),(z_v))$ as the free $F$-module on all spanning trees of $B(\md)$,
and define $\Psi$ by
\beg{epsip}{\Psi(T)=\sum_{T^\prime} \left(\frac{1}{1+\alpha(T,T^\prime)}+
\frac{1}{1+\beta(T,T^\prime)}\right)T^\prime.
}
Similarly as in our Comment earlier, the elements $\alpha(T,T^\prime)$ and $\beta(T,T^\prime)$
are not equal to $1$ (and hence \rref{epsip} makes sense) provided that
\beg{cass}{\parbox{3.5in}{The elements $u_f$ where $f$ ranges over all faces 
of $B(S^2)$ with one
face omitted, and the elements $z_v$ where $v$ ranges
over all vertices of $B(S^2)$ with one vertex omitted are jointly algebraically independent in $F$.}
\tag{C}}
(It would, in fact, suffice for the face and vertex variables in \rref{cass} to be {\em 
separately} algebraically independent, but the condition \rref{cass} as stated will
be more convenient for other purposes below.)

Let us now take this discussion one step further. Let $E$ be any field of characteristic $2$
containing an element $q_e$ for each edge
$e$ of $B(S^2)$. We can then set in $E$
\beg{efface}{u_f=\prod_{e} (q_{e})^{\alpha(e)}}
where 
$$d(f)=\sum_{e} \alpha(e) e,$$
and
\beg{evvertex}{z_v=\prod_{e}(q_e)^{\delta(v)(e)}.
}
Let $F$ be the subfield of $E$ generated by $u_f$, $z_v$. In the next Section, we shall prove
the following

\begin{proposition}
\label{elll}
The elements $q_e$ are algebraically independent in $E$ if and only if the Condition \rref{cass}
is satisfied in $F$ (and hence in $E$) . 
\end{proposition}

\noindent
{\bf Remark:} In the Novikov field $\Lambda$, the variables
$$q_e=T^{w(e)}$$
are algebraically independent, so Proposition \ref{elll} shows how 
$$C(\mathcal{D},E,(u_f),(z_v))$$ 
generalizes the cochain complex defined above by \ref{epsi}.
In this context, it is also worth noting that if $F$ is a subfield of $E$
and $C$ is a cochain complex of $F$-modules, then 
$$\rank_{E}(H^i(C\otimes_F E))=\rank_{F}(H^i(C))$$
(by flatness of field extensions). For this reason, from now on, we
shall work in general with complexes of the form $C(\mathcal{D},F,(u_f),(z_v))$
for a field $F$ satisfying the Condition \rref{cass}. To simplify notation,
we shall generally denote this complex simply by $C(\mathcal{D})$ 
where the field $F$ and the elements $u_f$, $z_v$ are understood.

\vspace{3mm}
Let us now state our main result:

\vspace{3mm}

\begin{theorem}
\label{t1p}
Let $F$ be a field of characteristic $2$ with elements $u_f$, $z_v$ satisfying
the relations \rref{erel1}, \rref{erel2} and the condition \rref{cass}. 
Then for each $i$,
\beg{erank}{\rank_F(H^i(C(\md,F,(u_f),(z_v)),\Psi)))}
defined by \rref{epsip} is independent of the choice of such a $F$, and of the projection
$\md$ of an oriented link $L$, subject to the condition \rref{lass}. If, 
further, $L$ is a knot, then \rref{erank} is independent of orientation. If
$L$ is a link which has a projection with more than $1$ connected component (i.e.
a split link),
then \rref{erank} is equal to $0$. 
\end{theorem}

\vspace{3mm}

\begin{figure}
\begin{tikzpicture}[line width=2pt]

\fill (.75,2.5) circle (3pt);
\fill (.75,-.13) circle (3pt);
\fill (3,.65) circle (3pt);

\draw[style=dashed] (.75,2.5) .. controls (1,2) and (1,.5) .. (.75,-.13);
\draw[style=dashed] (.75,-.13) .. controls (1.5,.4) and (2.5,.6) .. (3,.65);

\draw (.75,2.5) .. controls (0,1.75) and (0,0) .. (1,-.25) .. controls (2.25,-.5) and (2.75,.3) .. (3,.65);

\draw (.75,2.5) node[anchor=south]{$x$};
\draw (.75,-.13) node[anchor=north east]{$z$};
\draw (3,.65) node[anchor=south west]{$y$};

\fill (4.75,2.5) circle (3pt);
\fill (4.75,-.13) circle (3pt);
\fill (7,.63) circle (3pt);

\draw[style=dashed] (4.75,2.5) .. controls (5,2) and (5,.5) .. (4.75,-.13);
\draw[style=dashed] (4.75,-.13) .. controls (5.5,.4) and (6.5,.6) .. (7,.65);

\draw (7,.63) .. controls (7,2.25) and (5.5,3) .. (4.75,2.5) .. controls (4,1.75) and (4,0) .. (4.75,-.13);

\draw (4.75,2.5) node[anchor=south]{$x$};
\draw (4.75,-.13) node[anchor=north east]{$z$};
\draw (7,.65) node[anchor=south west]{$y$};
\draw (4.8,1.2) node[anchor=east]{$u$};
\draw (5.5,1.5) node[anchor=north west]{$v$};

\end{tikzpicture}
\caption{}
\end{figure}

Note that in view of the above Remark, 
one statement of the Theorem is already clear, namely the independence of the quantity
\ref{erank} of the field $F$ subject to the condition \rref{cass}: For, we may as well
work in the field of rational functions $\mathbb{F}_2(u_f,z_v)$ where $f$ run through all
but one face and $v$ run through all but one vertex of the black graph; the rank won't
change upon extension of fields. 
The following three sections consist of work toward the proof of Theorem \ref{t1p}. 
We conclude this section with a proof of Lemma \ref{lpsipsi}. In fact, in view of
the observations we made, it is more natural to prove the following generalization:

\begin{lemma}
\label{lpsipsi1}
Let $F$ be any field of characteristic $2$ with variables $u_f$ and $z_v$ for which
the expression \rref{epsip} makes sense (i.e. the denominators are non-zero). Then
we have \rref{epsipsi}.
\end{lemma}

\vspace{3mm}
\Proof
Our aim is to compute
\beg{epsipsit}{\Psi\Psi(T)}
for a spanning tree $T$, and prove that its coefficient on
any tree $T^{\prime\prime}$ is equal to $0$.
The key observation is that it actually suffices to consider the
case when $T$ has only two edges of height $0$, since otherwise
we may contract each component of the complement of the two open edges
in $T$ to a point and obtain the same coefficient. 

Now up to isomorphism, there is only one tree with two edges. It has vertices
$x,y,z$ and edges $\{x,z\}$, $\{y,z\}$ (of height $0$).
Then there are two non-isomorphic
choices of the tree $T^{\prime\prime}$ (consisting of two edges of height
$1$): the edges of $T^{\prime\prime}$ of height $1$
may be either $\{x,z\}$, $\{y,z\}$,
or $\{x,z\}$, $\{x,y\}$. (To clarify, in both cases we are dealing with a multigraph of $4$
edges here; we continue using our convention of writing
edges as ``sets'' because they are distinguished
by height. See Figure 3 where height $0$ edges are
rendered as dashed and height $1$ edges as solid.) 

 In the first case, the coefficient of \rref{epsipsit}
at $T^{\prime\prime}$ is a sum of two equal terms (each a product of two 
terms in opposite orders), so the sum is $0$ since we are in characteristic $2$.

The second case is non-trivial. Assuming, without loss of generality, that
the edges $\{x,y\}$, $\{y,z\}$ and the edge $\{x,z\}$ of height $0$ form 
a face $v$, and if the other bounded face $u$ is bounded by the two 
$\{x, z\}$ edges, then (identifying vertices and faces with
their corresponding variables), the formula we need to prove is
\beg{euxyv}{\begin{array}{l}
\displaystyle\left(\frac{1}{1+u}+\frac{1}{1+x}\right)
\left(\frac{1}{1+uv}+\frac{1}{1+y}
\right)+\\[4ex]
\displaystyle\left(\frac{1}{1+v^{-1}}+\frac{1}{1+x}\right)
\left(\frac{1}{1+uv}+\frac{1}{1+xy}
\right)+\\[4ex]
\displaystyle\left(\frac{1}{1+v}+\frac{1}{1+y}
\right)\left(\frac{1}{1+u}+\frac{1}{1+xy}\right)
=0
\end{array}
}
(the left hand side being the coefficient of \rref{epsipsit} at $T^{\prime\prime}$).
To verify \rref{euxyv}, notice that
$$\frac{1}{1+u}\frac{1}{1+uv}+\frac{v}{1+v}\frac{1}{1+uv}+\frac{1}{1+u}\frac{1}{1+v}=0,$$
$$\frac{1}{1+x}\frac{1}{1+y}+\frac{1}{1+x}\frac{1}{1+xy}+\frac{1}{1+xy}\frac{1}{1+y}=
\frac{1}{1+xy},$$
$$\begin{array}{l}
\displaystyle \frac{1}{1+u}\frac{1}{1+y}+\frac{1}{1+x}\frac{1}{1+uv}+\frac{1}{1+v^{-1}}
\frac{1}{1+xy}+\\[4ex]
\displaystyle 
\frac{1}{1+x}\frac{1}{1+uv}+\frac{1}{1+u}\frac{1}{1+y}+\frac{1}{1+xy}\frac{1}{1+v}=
\frac{1}{1+xy}.
\end{array}
$$
\qed

\vspace{3mm}

\section{The fundamental lemma}

\label{efund}

Recall that we assume \rref{lass}.
For a finite CW-complex $X$, note that
we have a canonical isomorphism between cellular chains and cellular cochains:
\beg{eccoch}{\diagram C_{k}^{cell}(X,\R)\rto^\cong
& C^{k}_{cell}(X,\R)
\enddiagram
}
which sends 
$$\sum_i \lambda_i e_i$$
for cells $e_i$ to the cochain whose value, on a $k$-cell $e$, 
is 
$$\sum_{i:e_i=e} \lambda_i.$$
We will treat this isomorphism as an identification. Note that such an identification
works over any subfield of $\R$, in particular over $\Q$. It does not, of course,
in general send cycles to cocycles, but it is important to note that it is 
independent of choice of orientation of cells (provided that we choose
the same orientation in homology and cohomology).

\label{slem}

\begin{lemma}
\label{l1}
(The Fundamental Lemma.) Suppose 
\beg{l1e1}{c:=\sum_{e_i\in E(\md)}\lambda_i e_i\in Z^{1}_{cell}(B(S^2),\R)
}
and also
\beg{l1e2}{
\tau c:=\sum_{e_i\in E(\md)}\lambda_i\tau( e_i)\in Z^{1}_{cell}(W(S^2),\R).
}
Then 
$$c=0\in C^{1}_{cell}(B(\mathcal{D}),\R).$$
\end{lemma}

\vspace{3mm}
\Proof
We have $H^{1}_{cell}(B(S^2),\R)=0$, so by \rref{l1e1},
there exists a function
$u:VB(\md)\r\R$ such that
\beg{l1e3}{\delta u=c.
}
Now the condition \rref{l1e2}, using \rref{l1e3}, translates to the equations
\beg{l1e4a}{\sum_{y:\{y,x\}\in EB(\md)} (u(y)-u(x))=0\;\text{for $x\in VB(\md)$},}
or 
\beg{l1e4}{u(x)=\frac{1}{\#(S_x)}\sum_{e\in S_x} u(y_e)
}
where 
$$S_x=\{e\in EB(\md)\;|\;\text{$e$ has vertices $x,y$}\}$$
and $e$ has vertices $x$ and $y_e$.
(The key observation is that, as one checks from the definitions, the summands of \rref{l1e4a} do not
change signs in dependence on orientation of edges.
Note also that \rref{l1e4} can be interpreted as a discrete analogue of $u$ being
``harmonic''.) 

Now \rref{l1e4} implies that $u$ is constant on connected components $C$ of
$B(\md)$ (actually, by our assumption, $B(\md)$ is connected). To see this,
consider
$$m_C=\min_{x\in VC} u(x).$$
By induction, we see that $u(y)=m_C$ for all $y\in VC$. This implies that
$c=\delta u=0$.
\qed

\vspace{3mm}
\begin{corollary}
\label{cl1}
The map 
\beg{emmap1}{\begin{array}{l}d\oplus\delta:C^{2}_{cell}(B(S^2),\Q)\oplus C_{0}^{cell}(B(S^2),\Q)\r
\\
\r C_{1}^{cell}(B(S^2),\Q)\cong C^{1}_{cell}
(B(S^2),\Q)
\end{array}}
is onto.
\end{corollary}

\Proof
By Lemma \ref{l1}, the kernel of the map \rref{emmap1} is 
$$Z^{2}_{cell}(B(S^2),\Q)\oplus Z_{0}^{cell}(B(S^2),\Q)\cong \Q^2.$$
Thus, the dimension of its image is equal to the number of faces plus number of vertices 
minus $2$, which is equal to the number of edges by the fact that the Euler characteristic
of $S^2$ is $2$.
\qed

\vspace{3mm}

\noindent
{\bf Proof of Proposition \ref{elll}:}
By Corollary \ref{cl1}, there exist natural numbers $n_e$
such that the field $F$ contains 
\beg{eppowers}{(q_{e})^{n_e}}
for each edge $e$ of $B(S^2)$. By assumption, the variables \rref{eppowers} are
algebraically independent, so the transcendence degree of $F$ over $\mathbb{F}_2$
is at least equal to the number of edges of $B(S^2)$, which is equal to the number
of the variables $u_f$ and $z_v$ with one face and one vertex omitted. Therefore,
those variables must all be algebraically independent (and in fact, equality in the
transcendence degree must arise).
\qed

\vspace{3mm}
\noindent
{\bf Remark:} It would be interesting to know if the assumptions of
Theorem \ref{t1p} regarding algebraic independence of variables
can be further weakened. For example, Baldwin and Levine in their
paper \cite{bl} are able to work over any variables which do not
satisfy a certain specific relation, which allows them ultimately to
work over the field of rational functions in a single variable.
This would be very interesting to know also in our present
setting, since it would make BOS cohomology much more computable.
Unfortunately, the only result that is easily seen in the present
setting is the following

\vspace{3mm}
\begin{proposition}
\label{generic}
Let $K$ be any field of characteristic $2$ with elements $u_{f}^{\prime}$, $z_{v}^{\prime}$ such that
the expression \rref{epsip} makes sense with $u_f$ replaced by $u_{f}^{\prime}$
and $z_v$ replaced by $z_{v}^{\prime}$ (i.e. the denominators are non-zero). Denote
this expression by $\Psi^\prime$. Let, further,
$F$ with elements $u_f$, $z_v$ satisfy the Condition \rref{cass}.  Then we have
\beg{egeneric}{\begin{array}{l}\rank_F(H^i(C(\mathcal{D},F,(u_f),(z_v)),\Psi))\leq \\
\rank_K(H^i(C(\mathcal{D},K,(u_{f}^{\prime}),(
z_{v}^{\prime})),\Psi^\prime)).\end{array}}
\end{proposition}

\Proof
We invoke the method of {\em cancellation} for computing 
cohomology of (co)chain complexes of vector spaces over a field $F$, which
was communicated to the first author by John Baldwin: Let us suppose $C$ is
a cochain complex with (homogeneous) basis $B$. Draw an oriented graph $\Gamma$
whose set of vertices is $B$ and there is an edge from $x$ to $y$ if the coefficient
of the differential from $x$ to $y$ is non-zero. We decorate the edge by the coefficient
of the differential, which we denote by $c_{(x,y)}$. By definition, $c_{(x,y)}$ is zero
if and only if there is no edge from $x$ to $y$.

Now a step of cancellation is performed by considering an oriented edge $(x,y)$ in $\Gamma$.
Then modify our data by erasing the vertices $x$, $y$ (and all adjacent edges) and for
pair of edges $(z,y)$ and $(x,t)$,
subtract the quantity $c_{(z,y)}c_{(x,y)}^{-1}c_{(y,t)}$ from $c_{z,t}$ (note that this
may involve erasing or creating an edge). Then the resulting complex has isomorphic
cohomology by an easy short exact sequence argument. When no more edges
are left, we have a basis of a vector space isomorphic to the cohomology of $C$.

Now in our situation, let us draw side by side the graphs $\Gamma_F$, $\Gamma_K$ of the cochain complexes
$C(\mathcal{D},F,(u_f),(z_v),\Psi)$ and $C(\mathcal{D},K,(u_{f}^{\prime}),(
z_{v}^{\prime}),\Psi^\prime)$. Write the coefficients $c_{(x,y)}$ of $\Gamma_F$ and
$\Gamma_K$ in terms of the
variables $u_f, z_v$ and $u_{f}^{\prime}$, $z_{v}^{\prime}$, respectively.
Then it follows from our assumptions that
whenever there is an edge $(x,y)$ in $\Gamma_K$, there is a corresponding edge in $\Gamma_F$.
This is because when a rational function in $z_1,...,z_n$ is defined and non-zero for some elements $z_1,...,z_n$
of a field of characteristic $2$, then it is also (defined and) non-zero in the field
$\F_2(z_1,\dots,z_n)$.
Therefore, we can perform cancellation on the edge $(x,y)$ in both graphs. Eventually, we will be left
with a situation where the graph obtained from $\Gamma_K$ has no edges (while the
graph obtained from $\Gamma_F$ may or may not have edges. The statement of the
Proposition follows.
\qed

\vspace{5mm}
\section{Reidemeister 1 and 2}
\label{sr2}

\vspace{3mm}

Next, we shall prove that BOS cohomology is invariant
under the three Reidemeister moves (see Figure 4). 
Note first that if a generic projection
$\md^\prime$ is obtained from a generic projection $\md$ by performing
a Reidemeister 1 move creating a new crossing, then we either added a
new vertex $v$ and an edge $e$ originating in $v$ to the black graph, or
a loop $f$ without adding a new vertex. The complex
$(C(\md^\prime),\Psi)$ is therefore isomorphic to $(C(\md),\Psi)$
up to shift of degrees. To compute the shift of degrees, note
that height of corresponding states
increases by $1$ if and only if $e$ has height $1$ or $f$ has
height $0$; otherwise, heights of corresponding states stay the same as
in $C(\md)$. However, by our conventions, the first case arises if and only if
the new crossing was negative. Thus, by the formula \rref{edegree}, the degree
of corresponding states remains unchanged in either case.

\vspace{3mm}

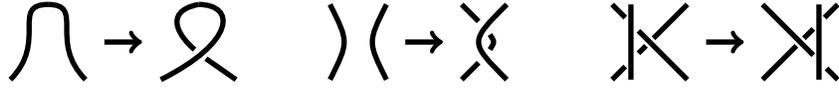
\begin{figure}
\label{freide}
\begin{tikzpicture}[line width=2pt]
\draw (0,0) .. controls (.5,.5) and (0,1) .. (.5,1);
\draw (.5,1) .. controls (1,1) and (.5,.5) .. (1,0);

\draw[->] (1.25,.5) -- (1.75,.5);

\draw (2,0) .. controls (3,.5) and (3,1) .. (2.5,1);
\draw (2.5,1) .. controls (2,.9) and (2.275,.540625) .. (2.45,.4);
\draw (2.6,.279462857) .. controls (3,0) .. (3,0);

\draw (4.25,1) .. controls (4.5,.5) .. (4.25,0);
\draw (5,1) .. controls (4.75,.5) .. (5,0);

\draw[->] (5.25,.5) -- (5.75,.5);

\draw (6,1) .. controls (6.1,.9) .. (6.225,.775);
\draw (6.375,.6) .. controls (6.44,.5) .. (6.375,.4);
\draw (6,0) .. controls (6.1,.1) .. (6.225,.225);

\draw (6.6,1) .. controls (6.1,.5) .. (6.6,0);

\draw (8.25,0) -- (8.25,1);
\draw (8,1) -- (8.175,.825);
\draw (8.35,.65) -- (9,0);
\draw (8,0) -- (8.175,.175);
\draw (8.35,.35) -- (8.45,.45);
\draw (8.55,.55) -- (9,1);

\draw[->] (9.25,.5) -- (9.75,.5);

\draw (10.75,1) -- (10.75,0);
\draw (10,1) -- (10.675,.325);
\draw (10.85,.15) -- (11,0);
\draw (10,0) -- (10.45,.45);
\draw (10.55,.55) -- (10.675,.675);
\draw (10.825,.825) -- (11,1);

\end{tikzpicture}
\caption{The three Reidemeister moves 1,2, and 3, respectively}
\end{figure}

\vspace{3mm}

Let us now turn to the Reidemeister 2 move.
Let $\md^\prime$ be the projection after a Reidemeister 2 move. By the isomorphism
of black and white complexes, we may assume that the number of black vertices
increases by 2. More precisely, there exists a vertex
$$u\in VB(\md)$$
such that
$$VB(\md^\prime)=(VB(\md)\smin \{u\})\amalg\{u_1,u_2,v\}.$$
Additionally, if $S$ is the set of all edges in $EB(\md)$ adjacent to $u$,
there exists a decomposition
$$S=S_1\amalg S_2$$
such that for every edge in $S_i$ with vertices $u,w$, there is an edge in
$EB(\md^\prime)$ with vertices $u_i,w$, $i=1,2$. Additionally, for every edge
$e\in EB(\md)$ neither vertex of which is $u$, $e\in EB(\md^\prime)$, and
we also have edges
$$\{u_i,v\}\in EB(\md^\prime)$$
where $\{u_i,v\}$ has height $i-1$. Finally, $EB(\md^\prime)$ contains no
other edges other than specified above (see Figure 5 - thick solid means
height $1$, thick dashed means height $0$, thin dashed means
unspecified height).  Note that we have a bijection
$$\phi:EB(\md)\r EB(\md^\prime)\smin\{\{u_1,v\},\{u_2,v\}\}$$
which sends $\{z,t\}$ to itself for $z,t\neq u$ and $\{z,u\}$ to the
appropriate $\{z,u_i\}$. Furthermore, $\phi$ preserves height. The main
purpose of this section is to prove the following

\vspace{3mm}

\begin{figure}
\begin{tikzpicture}

\fill (-.75,1) circle (2.5pt);
\fill (0,1) circle (2.5pt);
\fill (.75,1) circle (2.5pt);
\fill (0,0) circle (2.5pt);
\fill (-.75,-1) circle (2.5pt);
\fill (0,-1) circle (2.5pt);
\fill (.75,-1) circle (2.5pt);

\draw (0,0) node[anchor=east]{$u$};
\draw (-1,1.25) node[anchor=east]{$S_1$};
\draw (-1,-1.25) node[anchor=east]{$S_2$};

\draw[loosely dotted][line width=1.5pt] (-.5,1.5) -- (.5,1.5);

\draw[dashed] (-.75,1) -- (0,0) -- (.75,-1);
\draw[dashed] (0,1) -- (0,0) -- (0,-1);
\draw[dashed] (.75,1) -- (0,0) -- (-.75,-1);

\draw[loosely dotted][line width=1.5pt] (-.5,-1.5) -- (.5,-1.5);

\fill (4,1) circle (2.5pt);
\fill (4,0) circle (2.5pt);
\fill (4,-1) circle (2.5pt);
\fill (3.25,2) circle (2.5pt);
\fill (4.75,2) circle (2.5pt);
\fill (3.25,-2) circle (2.5pt);
\fill (4.75,-2) circle (2.5pt);
\fill (4,2) circle (2.5pt);
\fill (4,-2) circle (2.5pt);

\draw (4,1) node[anchor=east]{$u_1$};
\draw (4,0) node[anchor=east]{$v$};
\draw (4,-1) node[anchor=east]{$u_2$};
\draw (3,2.25) node[anchor=east]{$S_1$};
\draw (3,-2.25) node[anchor=east]{$S_2$};

\draw[loosely dotted][line width=1.5pt] (3.5,2.5) -- (4.5,2.5);

\draw[dashed][line width=2pt] (4,1) -- (4,0);
\draw[line width=2pt] (4,0) -- (4,-1);
\draw[densely dashed] (3.25,2) -- (4,1);
\draw[densely dashed] (4.75,2) -- (4,1);
\draw[densely dashed] (4,2) -- (4,1);
\draw[densely dashed] (3.25,-2) -- (4,-1);
\draw[densely dashed] (4.75,-2) -- (4,-1);
\draw[densely dashed] (4,-2) -- (4,-1);

\draw[loosely dotted][line width=1.5pt] (3.5,-2.5) -- (4.5,-2.5);

\draw (0,-3) node[anchor=north]{$B(\mathcal{D})$};
\draw (4,-3) node[anchor=north]{$B(\mathcal{D}')$};

\end{tikzpicture}
\caption{}
\end{figure}

\begin{proposition}
\label{pr2}
The chain complexes $(C(\md),\Psi)$, $(C(\md^\prime), \Psi)$, taken over fields
satisfying the assumption \rref{cass} for the respective link projections, have
cohomology groups of equal rank. 
\end{proposition}

\vspace{3mm}

To begin, there are two spectral sequences we may use to study
the complex $C(\md^\prime)$
(corresponding to two different decreasing filtrations).
We will discuss both, 
referring to them as the ``First'' and ``Second'' spectral sequence.
The First spectral sequence applies uniformly
to the Reidemeister 2 and Reidemeister 3 moves. In the case of the Reidemeister 2 move,
the Second spectral sequence
leads more clearly to the solution of the problem.
In the case of the Reidemeister 3 move, however,
the Second spectral sequence is not visible directly; we will use an analogue of
the First spectral sequence to reduce the problem to a situation where an
analogue of the Second
spectral sequence applies.

At this point, we assume we are working over any field of characteristic $2$
with elements $u_f$ and $z_v$
with respect to the projection $\md^\prime$
for which the differential \rref{epsip} makes sense (meaning that the denominators
are nonzero).

\vspace{3mm}
\noindent
{\bf The First spectral sequence:}
Denote by $F^pC(\md^\prime)$ the free $\Lambda$-module on all Kauffman states
of $\md^\prime$ (=spanning trees $T$ of $B(\md^\prime)$) for which 
\beg{r2e1}{
\sum_{e\in EB(\md):\phi(e)\in ET} h(e) +\sum_{e\in EB(\md):\phi(e)\notin ET}
(1-h(e))\geq p.
}
(Note that the left hand side is the formula for $h(T)$ modified by
{\em excluding} the terms for the edges $\{u_i,v\}$.) Then
$$\Psi F^pC(\md^\prime)\subseteq F^pC(\md^\prime)$$
(as the differential never decreases the height contribution of any single edge
of the black graph). 

Let us consider the spectral sequence associated with the filtration $F^p$.
To identify it, we need some additional notation. Let
$$K_i=\{T\in K(\md^\prime)\;|\;\{u_i,v\}\in ET,\;\{u_{2-i},v\}\notin ET\},$$
$$L=\{T\in K(\md^\prime)\;|\;\{u_i,v\}\in ET,\;i=1,2\}.$$
We have a bijection 
$$\diagram\kappa:K_1\rto^\cong & K_2\enddiagram$$
with
$$E(\kappa(T))=(E(T)\smin\{\{u_1,v\}\})\cup \{\{u_2,v\}\}.$$
Then clearly, for $T\in K_1$,
$$d_0(T)=c_T\kappa(T),\; c_T\neq 0\in \Lambda,$$
while for $T\in L$,
$$d_0(T)=0.$$
Now note that we have a canonical bijection
$$\iota:K(\md)\r L,$$
$$E\iota(T)=ET\cup \{\{u_1,v\},\{u_2,v\}\}.$$
Noting carefully that $\iota$ raises height by $1$, but the number
of negative crossings of $\md^\prime$ is also greater by $1$ than
the number of negative crossings of $\md$, we see that we have an isomorphism
of graded modules
\beg{ee1r2}{E_1\cong C(\md).}
However, we need to understand the bigrading. To this end, simply note that
the filtration degree of $\iota(T)$ is twice its degree minus $1$ plus the number of
negative crossings of $\md^\prime$, in other words,
\beg{r2ee1}{\text{$E_{1}^{p,q}\neq 0$ implies $p=2(p+q)+n_-(\md)$.}}
Thus, the $E_1$-term lies on a line of slope $-1/2$, and it follows that
the only possible differential is $d_2$, and the spectral sequence collapses
to $E_3$. 

Clearly, we can compute $d_2$, but it is important to note that despite 
the suggestive formula \rref{ee1r2}, $d_2$ is {\em not} simply an obvious modification
of $\Psi_\md$ by changing the field $F$: this is because of the
fact that a $d_2$ in a spectral
sequence associated to a decreasing filtration of a cochain complex is not computed
simply by applying the differential $d$ to a cocycle $c$ of the complex
of filtration degree $p$, even when $d_1$
is trivial: one must add
a counter-term to eliminate the summand of $d(x)$ in filtration
degree  $p+1$. In the present case, the differential $d_2$ is more cleanly computed
by the Second spectral sequence, which, in fact, leads directly to a stronger result.

\vspace{3mm}
\noindent
{\bf The Second spectral sequence:} Introduce another decreasing filtration $G^p$ on $C(
\md^\prime)$ defined by
\beg{r2e2}{
\sum_{e\in (EB(\md^\prime)\smin\phi(EB(\md)))\cap ET} h(e) +\sum_{e\in (EB(\md^\prime)
\smin\phi(EB(\md)))\smin ET}
(1-h(e))\geq p.
}
Roughly speaking, then, in the $G$-filtration, we are counting height contributions
of the edges $\{u_i,v\}$, i.e. exactly the edges not counted in the $F$-filtration.
We see that 
\beg{e3r2}{G^0(C(\md^\prime))=C(\md^\prime),\; G^3(C(\md^\prime))=0,
}
Thus, the associated graded cochain complex is non-trivial only in filtration
degrees $0,1,2$.

\vspace{3mm}

Let us begin by studying this situation in complete generality, and gradually
add information specific to $C(\md^\prime)$. Therefore, let us first consider
a cochain complex $\tilde{Q}$ with a decreasing filtration $G^i$ where
$G^0\tilde{Q}=\tilde{Q}$, $G^3 \tilde{Q}=0$. Let us denote the associated
graded pieces in filtration degrees $0,1,2$ by $U_0$, $Q$, $U_2$,
respectively. We shall assume we are working
in the category of $F$-modules where $F$ is a field of characteristic $2$.
In particular, since $F$ is a field, we may choose (arbitrarily)
splittings of the maps $G^i(\tilde{Q})\r G^i(\tilde{Q})/G^{i+1}(\tilde{Q})$.
After this choice, we see that the most general form $\tilde{Q}$ can take
is expressed in the following diagram:
\beg{euqu}{\diagram
U_0\drto^i\ddto_\eta&\\
&Q\dlto^j\\
U_2.&
\enddiagram
}
Here $U_0$, $Q$, $U_2$ are considered cochain complexes by the differential
on the associated graded pieces of $\tilde{Q}$, and the total differential 
is the sum of that differential and all applicable arrows of \rref{euqu}
(roughly, but note that not exactly, a totalization of a double complex).
The necessary and sufficient condition for this to work is that $i$, $j$ be
chain maps, and $\eta$ be a chain homotopy between $ji$ and $0$. 
Our convention (which we hope to justify later) is to denote the differentials
on $U_0$ and $Q$ by $d$, and the differential on $U_2$ by $d^\prime$, so the
homotopy condition reads
\beg{ethetah}{d^\prime\eta+\eta d=ji.
}
In the most general situation thus described, little can be said beyond
the spectral sequence associated with the filtration.

However, assume now also that
\beg{espec1}{\text{$\eta$ is an isomorphism of $\tilde{F}$-modules.}
}
With this special condition, we see immediately from \rref{ethetah} that
\beg{edprime}{d^\prime=\eta d\eta^{-1}+ji\eta^{-1}
}
(keep in mind that we are in characteristic $2$). However, we can say even
more:

\begin{lemma}
\label{ltheta}
Under the assumption \rref{espec1}, 
\beg{edeform}{d+i\eta^{-1}j} 
is a differential on $Q$, and $\tilde{Q}$ is quasiisomorphic to $(Q,d+i\eta^{-1}j)[1]$
(the square bracket denotes degree shift by the specified number).
\end{lemma}

\Proof
We have
$$\begin{array}{l}(d+i\eta^{-1}j)(d+i\eta^{-1}j)=dd+di\eta^{-1}j+
i\eta^{-1}jd+i\eta^{-1}ji\eta^{-1}j=\\
di\eta^{-1}j+
i\eta^{-1}jd+i \eta^{-1}\eta d\eta^{-1}j + i\eta^{-1}d^\prime j=0.
\end{array}
$$
A chain map 
$$\tau: (Q,d+i\eta^{-1}j)[1]\r \tilde{Q}$$
is defined by
$$\tau(x)=x+\eta^{-1}j(x).$$
To see that $\tau$ is a chain map, compute
$$\begin{array}{l}
\tau(dx+i\eta^{-1}jx)=
dx+\eta^{-1}jdx+i\eta^{-1}jx+\eta^{-1}ji\eta^{-1}j(x)=\\
dx+\eta^{-1}\eta d\eta^{-1}jx +\eta^{-1}ji\eta^{-1}jx=\\
dx+d\eta^{-1}jx,
\end{array}$$
while
$$\begin{array}{l}
d_{\tilde{Q}}\tau(x)=dx+i\eta^{-1}jx +d\eta^{-1}j(x) +\eta\eta^{-1}jx+jx=\\
dx+d\eta^{-1}jx.
\end{array}
$$
Clearly, additionally, the map $\tau$ is injective. We claim that its cokernel
is isomorphic to $\overline{U}$, which is the totalization of the double complex 
$$\diagram
U_{2}^{\prime}\dto^{\lambda}\\
U_2
\enddiagram
$$
where $\lambda$ is an arbitrary chain isomorphism, 
which is clearly acyclic. To this end, we construct a chain map
$$\mu:\tilde{Q}\r\overline{U},$$
given as the identity on $U_2$, by
$$\mu(x)
=\lambda^{-1}\eta(x)
$$
for $x\in U_0$, and by
$$\mu(x)=\lambda^{-1}jx
$$
for $x\in Q$. To verify that $\mu$ commutes with the
differential on
$x\in U_0$, we have:
$$\begin{array}{l}
d\mu(x)=d\lambda^{-1}\eta(x)=\eta(x)+\lambda^{-1}d^\prime\eta(x)=\\
\eta(x)+\lambda^{-1}\eta d x +\lambda^{-1}jix=\mu (dx +ix +\eta x).
\end{array}
$$
To verify $\mu$ commutes with the differential on $x\in Q$, we have 
$$\begin{array}{l}
d\mu(x)=\lambda^{-1}d^\prime jx+jx=\lambda^{-1}jdx+jx=\mu(dx+jx).
\end{array}
$$
Now obviously the sequence of cochain complexes
$$\diagram
0\rto &(Q,d+i\eta^{-1}j)[1]\rto^(.8){\tau} & \tilde{Q}\rto^{\mu} &\overline{U}\rto & 0
\enddiagram
$$
is exact, which implies our statement by the long exact sequence in cohomology.
\qed

\vspace{3mm}
In the case $\tilde{Q}=C(\md^\prime)$, we choose the splitting so that
$U_0$, $Q$, $U_2$ are generated by $K_1$, $L$, $K_2$ respectively. 
At this point, let us introduce more specific assumptions about
the field we are working in. Let $F$ be a field satisfying the assumption
\rref{cass} for the link projection $\md$. We shall now produce a field $\tilde{F}$
with variables $u_f$, $z_v$ corresponding to the link projection
$\md^\prime$ satisfying \rref{erel1} and \rref{erel2}. We will not necessarily
assume that condition \rref{cass} is satisfied with $F$ replaced by $\tilde{F}$, but
we will require that \rref{epsip} be defined (i.e. that the denominators be non-zero). 
The field $\tilde{F}$ is constructed as follows:
To $F$, we
adjoin
two variables $x,y$ corresponding to the edges $(u_1,v)$, $(v,u_2)$
(and also the corresponding white edges).
We will allow the possibility of an algebraic relation between $x$ and $y$,
as long as \rref{epsip} will make sense, but we will assume that each of the variables
$x$, $y$ is algebraically independent from $F$.
In specifying black vertex and face variables for $C(\md)$ in $\tilde{F}$, 
our convention is that the element of
$\tilde{F}$ associated to a face $f$ in $B(\md^\prime)$ is equal to
the element associated with the corresponding face in $B(\md)$ (obtained by 
contracting the new edges), times any of the variables $x,y$ (or their inverses)
corresponding to any of the new edges $f$ may contain with the appropriate
orientations. The element of $\tilde{F}$ associated with a vertex $v$
in $B(\md^\prime)$ is by our convention
the corresponding vertex variable of $B(\md)$, times the product
of any of the variables $x, y$ or their inverses corresponding to the  edges
adjacent to $v$ in $\md^\prime$. The element associated with
the vertex $v$ is $xy^{-1}$. Examples of an allowable choice for $x$, $y$ 
are either two variables jointly algebraically independent over $F$, or powers
$t^m$, $t^n$ of a variable $t$ algebraically independent of $F$, as long
as $m,n\neq 0$, $m\neq -n$.

With these conventions, if $T$ is a spanning tree in $K_1$, adding $\{u_2,v\}$ to 
$T$ specifies a black circuit. If we denote the element associated with the corresponding
black circuit in $B(\md)$ by $b$, then
\beg{efeta}{\eta(T)=\left(\frac{1}{1+bxy}+\frac{1}{1+xy^{-1}}\right)\kappa(T).
}
The coefficient is non-zero, so \rref{espec1} holds.

\vspace{3mm}
It is not difficult to see by definition that
\beg{eapp12}{\parbox{3.5in}{$Q$ with the differential \rref{edeform} is isomorphic
to the $E_2$-term of the First spectral sequence.}
}
Despite its aesthetic appeal, however, Lemma \ref{ltheta} still does not solve
our problem: even though we have a chain complex isomorphic to $Q$ as $\tilde{F}$-modules,
there is no a priori reason to suspect any connection between the differentials
$d$ and \rref{edeform}: it is not even reasonable to call \rref{edeform} a deformation
of $d$, since in general the two summands of \rref{edeform} do not commute.

Here is where we need to bring in even more concrete information from the situation
at hand. As a warm-up, let us consider the differential $d^\prime$ on $U_2$ 
(see \rref{edprime}) instead
of the differential \rref{edeform} on $Q$. Notice that both differentials are
of similar form, a sum of two terms, one of which is related to a differential
we know ($d$ on $U_0$ in case of \rref{edprime} and $d$ on $Q$ in case of \rref{edeform}),
and the other is expressed as a composition of the maps $i,j$. 

In the case of \rref{edprime}, however, we do have another way of understanding
the differential $d^\prime$: Recalling that $U_2$ is isomorphic to $U_0$ as a
$\tilde{F}$-module by the bijection $\kappa$, and recalling our conventions
regarding $\tilde{F}$, one can see that both the differentials $d$ on $U_0$ and 
$d^\prime$ on $U_2$ are in fact the Baldwin-Ozsv\'{a}th-Szab\'{o} differential
with different choice of variables for a suitable
link projection. Consider first the summands of $\Psi$ which 
relate only trees in $K_i$ with a fixed $i$. The difference of
coefficients is only in white cycles which cross the edges $\{u_i,v\}$;
for a white circuit $w$ in $K_1$ crossing the edge $\{u_2,v\}$ and
associated element $\alpha$ in $C(\md^\prime)$, the element associated
with the corresponding white circuit in $K_2$ will be $\alpha$ multiplied by
$xy^{-1}$ or $yx^{-1}$, depending on the orientation. Both of these are,
in fact, forms of $\Psi$ with different choices of variables in $\tilde{F}$
for the projection $\mathcal{E}$ obtained by performing a skein move
instead on the two arcs on $\md$ involved in the R2 move we are
studying - see Figure 6. (More explicitly, $B(\mathcal{E})$ is
obained from $B(\md^\prime)$ by deleting the vertex $v$ and the edges
$\{u_i,v\}$.)

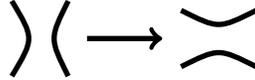
\begin{figure}[!ht]
\label{fskein}
\begin{tikzpicture}[line width=2pt]
\draw (-3,1) .. controls (-2.7,.5) .. (-3,0);
\draw (-2.25,1) .. controls (-2.5,.5) .. (-2.25,0);

\hspace{.1cm}
\draw (-.75,.125) .. controls (-.25,.375) .. (.25,.125);
\draw (-.75,.875) .. controls (-.25,.625) .. (.25,.875);

\draw[->] (-2,.5) -- (-1,.5);

\end{tikzpicture}
\caption{The skein move}
\end{figure}

Therefore, we know that
\beg{eisor2}{\parbox{3.5in}{If the variables $x$, $y$ are algebraically independent
over $F$, then the cohomology groups of $(U_0,d)$, $(U_2,d^\prime)$
have equal ranks.}} 
In fact, writing this down explicitly in terms of variables yields the following identity,
which will be useful later:

\begin{lemma}
\label{lakl}
In a field of characteristic $2$, we have the identity
\beg{eakl}{\begin{array}{l}\displaystyle\left(
\frac{1}{1+a}+\frac{1}{1+k}\right)^{-1}
\left(\frac{1}{1+a}+\frac{1}{1+\ell}\right)=\\
\displaystyle
\left(\frac{1}{1+\ell^{-1}k}+\frac{1}{1+k}\right)^{-1}
\left(\frac{1}{1+\ell^{-1}k}+\frac{1}{1+a^{-1}k}\right)
\end{array}
}
whenever the denominators are non-zero.
\end{lemma}

\Proof
Bringing the terms on the left hand side of \rref{eakl} to common multiplier,
we get
\beg{e1akl}{\frac{(1+k)(a+\ell)}{(1+\ell)(a+k)}.}
Doing the same on the right hand side gives
\beg{e2akl}{\frac{(1+k)(\ell^{-1}k+a^{-1}k)}{(1+a^{-1}k)(\ell^{-1}k+k)}.
}
Now \rref{e1akl} is gotten from \rref{e2akl} by dividing both
numerator and denominator by $a^{-1}\ell^{-1}k$.
\qed

\vspace{3mm}
Based on this, one could hope to apply an analogous principle to the differential
\rref{edeform} if we can somehow swap the roles of $Q$ and $U$. This, in fact, can
be done by considering an R2 move on the projection $\mathcal{E}$, and relating
all the new variables appropriately. Let $\mathcal{E}^\prime$ be the projection
obtained from $\mathcal{E}$ by an R2 move on the arcs related to the arcs of
$\md$ on which we performed the original R2 move by a skein move. Then 
$B(\mathcal{E}^\prime)$ is obtained from $B(\mathcal{E})$ by adding two new
edges $(u_1,u_2)$. Denote these edges by $e,f$ (see Figure 7, with the
same conventions as in Figure 5). Let, additionally, $h(e)=0$,
$h(f)=1$, and extend the field $\tilde{F}$ further into a field $F^\prime$
by attaching two new variables
$z,t$ associated with the edges $e$, $f$, respectively, each algebraically
independent of $F$. Our conventions regarding calculating the elements
$\alpha(T,T^\prime)$, $\beta(T,T^\prime)$ for $C(\mathcal{E}^\prime)$ are the
same as in the case of $C(\md^\prime)$: specifically, a black circuit $c$
in $B(\mathcal{E}^\prime)$ which does not contain any of the edges $e,f$ is
assigned the same element as in $C(\mathcal{E})$; if $c$ contains one or both
of the edges $e,f$, and if the corresponding circuit in $C(\md)$ (obtained by
contracting the edges $e,f$) is assigned an element $b$, then $c$ is assigned
the element $b$ multiplied by some of the elements $z$, $z^{-1}$, $t$, $t^{-1}$,
depending
on which of the edges $e$, $f$ $c$ contains, and
orientation. Regarding white circuits $w$, again, take the product of all the vertex
variables of $B(\md)$ inside (resp. outside) of $w$ depending on whether $w$ is
oriented counter-clockwise or clockwise, times, possibly, some of the elements
$z$, $z^{-1}$, $t$, $t^{-1}$, depending on which of the edges $e$, $f$ the circuit
$w$ crosses, and orientation. Our only assumption about the variables $z,t$ at this
point is that the differential \rref{epsip} make sense for the link projection
$\mathcal{E}^\prime$ (i.e. that the denominators be non-zero).

\begin{figure}
\begin{tikzpicture}

\fill (-.75,2) circle (2.5pt);
\fill (0,2) circle (2.5pt);
\fill (.75,2) circle (2.5pt);
\fill (0,1) circle (2.5pt);
\fill (-.75,-2) circle (2.5pt);
\fill (.75,-2) circle (2.5pt);
\fill (0,-2) circle (2.5pt);
\fill (0,-1) circle (2.5pt);

\draw (-1,2.25) node[anchor=east]{$s_1$};
\draw (0,1) node[anchor=east]{$u_1$};
\draw (0,-1) node[anchor=east]{$u_2$};
\draw (-1,-2.25) node[anchor=east]{$s_2$};

\draw[loosely dotted][line width=1.5pt] (-.5,2.5) -- (.5,2.5);

\draw[densely dashed] (-.75,2) -- (0,1);
\draw[densely dashed] (.75,2) -- (0,1);
\draw[densely dashed] (0,2) -- (0,1);
\draw[densely dashed] (-.75,-2) -- (0,-1);
\draw[densely dashed] (.75,-2) -- (0,-1);
\draw[densely dashed] (0,-2) -- (0,-1);

\draw[loosely dotted][line width=1.5pt] (-.5,-2.5) -- (.5,-2.5);

\fill (3.25,2) circle (2.5pt);
\fill (4,2) circle (2.5pt);
\fill (4.75,2) circle (2.5pt);
\fill (4,1) circle (2.5pt);
\fill (4,-1) circle (2.5pt);
\fill (3.25,-2) circle (2.5pt);
\fill (4.75,-2) circle (2.5pt);

\draw (3,2.25) node[anchor=east]{$s_1$};
\draw (4,1) node[anchor=east]{$u_1$};
\draw (3.7,0) node[anchor=east]{$e$};
\draw (4.3,0) node[anchor=west]{$f$};
\draw (4,-1) node[anchor=east]{$u_2$};
\draw (3,-2.25) node[anchor=east]{$s_2$};

\draw[loosely dotted][line width=1.5pt] (3.5,2.5) -- (4.5,2.5);

\draw[densely dashed] (3.25,2) -- (4,1);
\draw[densely dashed] (4.75,2) -- (4,1);
\draw[densely dashed] (4,2) -- (4,1);

\draw[line width=2pt] (4,1) .. controls (4.4,.5) and (4.4,-.5) .. (4,-1);
\draw[dashed][line width=2pt] (4,1) .. controls (3.6,.5) and (3.6,-.5) .. (4,-1);

\draw[densely dashed] (3.25,-2) -- (4,-1);
\draw[densely dashed] (4.75,-2) -- (4,-1);
\draw[densely dashed] (4,-2) -- (4,-1);
\fill (4,-2) circle (2.5pt);

\draw[loosely dotted][line width=1.5pt] (3.5,-2.5) -- (4.5,-2.5);

\draw (0,-3) node[anchor=north]{$B(\mathcal{E})$};
\draw (4,-3) node[anchor=north]{$B(\mathcal{E}')$};

\end{tikzpicture}
\caption{}
\end{figure}

Denote by $L^{\prime}_{1}$, resp. $L^{\prime}_{2}$ resp. $K^\prime$ the
sets of spanning trees of $B(\mathcal{E})$ which contain $e$ resp. $f$ resp.
neither $e$ nor $f$.
Now, analogously as above, filtering $C(\mathcal{E}^\prime)$ by
the total height contribution of the edges $e,f$ only, 
and performing the same analysis as we did for $C(\md^\prime)$, we see
that $C(\mathcal{E}^\prime)$ is isomorphic to a cochain complex
of the form
\beg{equq}{
\diagram
Q_0\drto^{j^\prime}\ddto_\xi&\\
& U\dlto^{i^\prime}\\
Q_2&
\enddiagram
}
where $Q_0$ resp. $Q_2$ resp. $U$ are generated by $L^{\prime}_{1}$ resp. 
$L^{\prime}_{2}$ resp. $K^\prime$. Once again, $\xi$ is an isomorphism of 
$F^\prime$-modules. Specifically if $T$ is a spanning tree in $L^{\prime}_{1}$
which, by deleting the edge $e$, creates a white circuit with associated
element $v$, and if 
$$\kappa^\prime:L^{\prime}_{1}\r L^{\prime}_{2}$$
is the canonical bijection (obtained by replacing the edge $e$ with the edge $f$),
then
\beg{efxi}{\xi(T)=\left(
\frac{1}{1+ztv}+\frac{1}{1+zt^{-1}}
\right)\kappa^\prime(T).
}
If we denote the differentials on $Q_0$, $U$ by $d$ and the differential on 
$Q_2$ by $d^{\prime\prime}$ (justified, again, by the idea that the last of
the three differentials must be distinguished while the others can be
understood from the context), \rref{edprime} translates to 
$$d^{\prime\prime}=\xi d\xi^{-1}+i^\prime j^\prime \xi^{-1},$$
or, equivalently,
\beg{edij}{\xi d\xi^{-1}= d^{\prime\prime}+i^\prime j^\prime \xi^{-1}.
}
Our strategy is to set up relations between the variables $x,y,z,t$
so that the right hand side of \rref{edij} is equal to \rref{edeform},
and 
\beg{edq}{\parbox{3.5in}{$d^{\prime\prime}=d_Q$ where $d_Q$ denotes
the original (unperturbed) differential on $Q$.}}
Then we know (analogously to \rref{eisor2}) that 
\beg{econclude}{\parbox{3.5in}{The differential \rref{edeform} is related to $\xi^{-1}d_Q\xi$
by a change of variables (i.e. by an automorphism of the field $F^\prime$).}
}

\vspace{3mm}
\begin{lemma}
\label{lxyzt}
The right hand side of \rref{edij} is equal to \rref{edeform} when 
\beg{esolution}{x=t^2,\; y=t^{-1},\; z=t^{-2}.}
\end{lemma}

\Proof
The equality we need will hold when
\beg{eiip}{i=i^\prime,
}
\beg{ejetaxi}{\eta^{-1}j=j^\prime\xi^{-1}.
}
Let first interpret \rref{edq}. Consider two spanning trees $T$, $T^\prime$ 
in $L$ where $T^\prime$ is obtained from $T$ by omitting an edge of
height $0$ and adding an edge of height $1$. Then the coefficient of $T^\prime$ in $d_Q(T)$
is 
$$\frac{1}{1+bxy}+\frac{1}{1+w}$$
where $w$ is the appropriate white circuit, while the coefficient of $d^{\prime\prime}$
between the corresponding trees in $L_2$ is
$$\frac{1}{1+bt} +\frac{1}{1+w},$$
so we see that \rref{edq} is satisfied provided that
\beg{eeexyt}{xy=t.}

Next, we impose the equality \rref{eiip}. When performing $i$ on a spanning tree
$T$, we delete an edge, thus creating a white circuit. Denote the corresponding
element of $F^\prime$ by $v$. Then the corresponding coefficient in $i$ is
$$\frac{1}{1+yvw}+\frac{1}{1+bxy},$$
and the corresponding coefficient in $i^\prime$ is 
$$\frac{1}{1+ztvw}+\frac{1}{1+bt}.$$
Thus, \rref{eiip} will hold if we impose
\beg{exyzt}{zt=y, xy=t,}
which subsumes \rref{eeexyt}.

Next, however, we must consider the equation \rref{ejetaxi}, which translates to
\beg{ejetaxip}{\begin{array}{l}
\displaystyle
\left(\frac{1}{1+bxy}+\frac{1}{1+xy^{-1}}
\right)^{-1}
\left(\frac{1}{1+bxy}+\frac{1}{1+vx}\right)=
\\[4ex]
\displaystyle
\left(\frac{1}{1+ztv}+\frac{1}{1+bz}\right)
\left( \frac{1}{1+ztv}+\frac{1}{1+zt^{-1}}
\right)^{-1}.
\end{array}
}
By \rref{eakl} of Lemma \ref{lakl}, (and by \rref{eklsign}), \rref{ejetaxip} holds when
$$\begin{array}{ll}
k=x^{-1}y & \ell=v^{-1}x^{-1}\\
ztv=\ell^{-1}k & k=z t^{-1}\\
a=b^{-1}x^{-1}y^{-1} & a^{-1}k=bz.
\end{array}
$$
This in turn holds under the assumption \rref{esolution}. Note that this also implies 
\rref{exyzt}.

To be completely precise, we have solved the ``non-trivial'' case of the equations
\rref{edq}, \rref{eiip}, \rref{ejetaxi}: There is another ``trivial'' case
when the black cycle does not go through the vertex $v$ (resp. any of the edges 
$e$, $f$). In this case, the corresponding components of the differentials $d_Q$,
$d$, $d^{\prime\prime}$, \rref{edeform} coincide (in particular, the
corresponding component of the $i\eta^{-1}j$ summand is $0$ and the corresponding
component of $d$ commutes with $\xi$, so the equations remain true 
in that case as well).
\qed

We have therefore proved \rref{econclude} under the assumption
\rref{esolution}. Note however that we are not quite done yet, since the field $\tilde{F}$ 
does not satisfy the assumption \rref{cass} (since we have introduced an
algebraic relation between the elements $x,y$).  To remedy this situation, 
we need to observe that the proof of Lemma \ref{lxyzt} in fact gives a stronger 
statement (which would have been awkward to state at first):

\begin{lemma}
\label{lxyzt1}
Consider two spanning trees $T$, $T^\prime$ in $L$ where $T^\prime$ is obtained
from $T$ by omitting an edge of height $0$ and adding an edge of height $1$, thus specifying
a black circuit $b$ and a white circuit $w$. Denote, additionally, by $v$ the white circuit
obtained by deleting the edge involved in the definition of $i(T)$. Then, using formula
\rref{esolution} as the definition of $x,y,z$, the coefficients of the right hand side
of \rref{edij} and \rref{edeform},  calculated by rewriting the formula \rref{epsip} for the BOS
differential in terms of black and white circuits, are equal as elements of the field
$$\Z/2(b,v,w,t).$$
\end{lemma}

\qed

\vspace{3mm}

We now claim

\begin{proposition}
\label{pxyzt}
There exists a field $F^{\prime\prime}\supset F^\prime=F(t)$
and embeddings of fields
\beg{eeiota}{\iota:F(x,y)\r F^{\prime\prime},}
$$\kappa:F\r F^{\prime\prime}$$
such that the image of the differential \rref{edeform} under $\iota$
and the image of the differential $\xi^{-1}d_Q\xi$ under $\kappa$ coincide.
\end{proposition}

\begin{proof}
Consider first the field $F^\prime=F(t)$ (with $x,y,z$ defined by \rref{esolution}). Now
let $s$ be a new variable algebraically independent of the rest. Now define $F^{\prime\prime}=F(s,t)$.
For a black circuit $b$ labelled in $F^\prime$, let
$$b^\prime =bs^{2\epsilon}$$
where $\epsilon=0$ if the black circuit does not pass through $v$, and $\epsilon=1$
(resp. $-1$) when the black circuit contains the oriented edge $(u_1,v)$, and hence
also $(v,u_2)$ (resp. $(v,u_1)$, and hence also $(u_2,v)$). Note that this
definition is multiplicative on black circuits (identifying, as before, a circuit with the product
of its edge variables or their inverses, depending on orientation).

Now by Lemma \ref{lxyzt1}, if we can find an embedding \rref{eeiota} such that, in computing
the BOS differential by formula \rref{epsip}, $b$ is replaced by $b^\prime$ by the embedding (while the
variables $v,w,t$ remain unchanged), and an embedding
$$\lambda:F^\prime\r F^{\prime\prime}$$
which also sends $b$ to $b^\prime$ while fixing $v,w,t$, then the image of \rref{edeform}
under $\iota$ is equal to the image of the right hand side of \rref{edij} under $\lambda$.

Regarding $\iota$, we may simply choose the identical embedding on $F$, while
sending
$$x\mapsto xs=t^2s,\;y\mapsto ys=t^{-1}s.$$
Note that these elements are algebraically independent over $F$.
Regarding $\lambda$,  this embedding will be identical on $t$ and on the vertex variables, while
its definition on face generators of $F$ is possible by the multiplicativity of $(?)^\prime$.

Thus, our statement follows from \rref{edij}.
\end{proof}

\vspace{3mm}

\section{Reidemeister 3}

\label{sr3}

The methods of the last section do not apply to the Reidemeister
3 move directly because the projections before and after an R3 move play symmetrical
roles: there is no obvious candidate inside the Baldwin-Ozsv\'{a}th-Szab\'{o}
complex of one projection for
a part which would be isomorphic to some modification of the complex of the
other. To get around this, we use the following idea suggested to us
by John Baldwin: let us study braids on three strands labelled, from left to 
right, $1,2,3$. Let $a$ resp. $b$ be the braid crossing strand $1$ over strand $2$
(resp. strand $2$ over strand $3$). Then the famous braid relation can be written
in the form
\beg{ebraid1}{aba^{-1}=b^{-1}ab,
}
which means that we have an unbraid
\beg{ebraid2}{b^{-1}a^{-1}baba^{-1}
}
(see Figure 8).
\begin{figure}
\label{funbraid3}
\begin{tikzpicture}
\braid[braid colour=black,strands=3,braid start={(0,0)}]
{\sigma_{1}^{-1} \sigma_{2} \sigma_{1} \sigma_{2} \sigma_{1}^{-1} \sigma_{2}^{-1}}
\draw (1,0) node[anchor=south] {1};
\draw (2,0) node[anchor=south] {2};
\draw (3,0) node[anchor=south] {3};
\draw[->,very thick] (1,-6) -- (1,-6.1);
\draw[->,very thick] (2,-6) -- (2,-6.1);
\draw[->,very thick] (3,-6) -- (3,-6.1);
\end{tikzpicture}
\caption{The $b^{-1}a^{-1}baba^{-1}$ unbraid}
\end{figure}
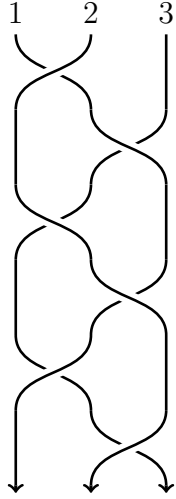

Now consider a generic projection $\md$ with three arcs $1,2,3$ such that 
$1$, $2$ bound a component of $S^2\smin \md$ labelled black, and 
$2$, $3$ bound a component of $S^2\smin \md$ labelled white. Then by a BR move
we shall mean an operation where we replace the arcs $1,2,3$ by the unbraid \rref{ebraid2}.

\begin{proposition}
\label{pbraid}
Suppose a generic projection $\md^\prime$ of an oriented link $L$ (satisfying \rref{lass})
is obtained from a generic projection $\md$ of $L$ using the BR move. Then 
$H^i(C(\md^\prime))$ and $H^i(C(\md))$ have equal ranks.
\end{proposition}

\vspace{3mm}

\begin{corollary}
\label{ecorr3}
The rank of $H^i(C(\md))$ is invariant under the Reidemeister 3 move.
\end{corollary}

\vspace{3mm}
\noindent
{\bf Proof of Corollary \ref{ecorr3} using Proposition \ref{pbraid}:}  Simply note that
change from  $b^{-1}ab$ to $aba^{-1}$ is an R3 move (see Figure 9). 
If we want to make this move
inside a projection $\md$, first change $b^{-1}ab$ to 
\beg{eintermed}{b^{-1}abb^{-1}a^{-1}baba^{-1}}
using the BR move, and then change \rref{eintermed} to $aba^{-1}$ using a sequence of
R2 moves (which we can do by Proposition \ref{pr2}) - see Figure10), thus implying invariance
of BOS cohomology under the R3 move. (Note: We have included orientations in figures 8,9 and 10
in reference to composition of braids. However, this is not required to coincide with the
orientation of the link. In BOS cohomology, change of orientation only affects the grading shift
in R1 moves, not in R2 and R3 moves.)
\qed

\begin{figure}
\label{funbraidR3}
\begin{tikzpicture}[line width=1pt]
\braid[braid colour=black,strands=3,braid start={(-3,0)}]
{\sigma_1^{-1},\sigma_2,\sigma_1}
\draw[->] (-2,-3) -- (-2,-3.1);
\draw[->] (-1,-3) -- (-1,-3.1);
\draw[->] (0,-3) -- (0,-3.1);
\draw[->] (4,-3) -- (4,-3.1);
\draw[->] (5,-3) -- (5,-3.1);
\draw[->] (6,-3) -- (6,-3.1);

\draw (-2,0) node[anchor=south] {1};
\draw (-1,0) node[anchor=south] {2};
\draw (0,0) node [anchor=south] {3};
\draw (4,0) node [anchor=south] {1};
\draw (5,0) node [anchor=south] {2};
\draw (6,0) node [anchor=south] {3};

\draw[->] (1,-1.5) -- (3,-1.5);
\braid[braid colour=black,strands=3,braid start={(3,0)}]
{\sigma_2,\sigma_1,\sigma_2^{-1}}
\end{tikzpicture}
\caption{Reidemeister 3 is equivalent to replacing the $aba^{-1}$ braid with the $b^{-1}ab$ braid}
\end{figure}
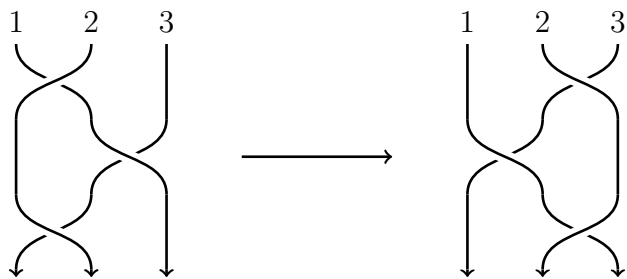

\vspace{3mm}

\begin{figure}
\label{funbraidR2}
\begin{tikzpicture}[line width=1pt]

\draw[->] (-2,0) -- (-2,-2.1);
\draw[->] (-1,0) -- (-1,-2.1);

\draw[->] (-.5,-1) -- (.5,-1);

\draw[->] (1,-2) -- (1,-2.1);
\draw[->] (2,-2) -- (2,-2.1);

\draw (-2,0) node[anchor=south] {1};
\draw (-1,0) node[anchor=south] {2};
\draw (1,0) node[anchor=south] {1};
\draw(2,0) node[anchor=south] {2};

\braid[braid colour=black,strands=2, braid start={(0,0)}]
{\sigma_1^{-1},\sigma_1}

\end{tikzpicture}
\caption{Reidemeister 2 is equivalent to replacing the unbraid with the $a^{-1}a$ unbraid}
\end{figure}

\vspace{3mm}
The remainder of this section is dedicated to proving Proposition \ref{pbraid}. The method
of proof is, in fact, 
more or less analogous to the proof of Proposition \ref{pr2}, but unfortunately,
the situation is more complicated. Let $\md^\prime$ be a projection obtained
from a projection $\md$ by the BR move. By our definition of the BR move,
we know that $B(\md)$ has a vertex $u$ in the component of $S^2\smin \md$
shared by arcs $1$ and $2$. Thus, there is a white component shared by arcs
$2$ and $3$, and there is another face colored black adjacent to $3$, which 
corresponds to a vertex $w$ of $B(\md)$. 

\vspace{3mm}
Then we have
$$V(B(\md^\prime))=(V(B(\md))\smin \{u\})\amalg \{u_1,u_2,v_1,v_2\}.$$
To describe $E(B(\md^\prime))$, we note that, once again, if we denote by $S$
the set of edges adjacent to $u$ in $B(\md)$, then 
$$S=S_1\amalg S_2$$
such that to each edge $\{u,q\}\in S_i$ there corresponds, in $B(\md^\prime)$,
an edge $\phi(\{u,q\}):=\{u_i,q\}$. In addition, every edge $\{q,q^\prime\}$ 
of $B(\md)$ where $q,q^\prime\neq u$
is also present in $B(\md^\prime)$ (including the case when one or both of $q,q^\prime$
are equal to $w$), and the following additional special edges are also in 
$B(\md^\prime)$ (we choose orientations to make assignment of elements easier later):
$$\begin{array}{l}e_1=(u_1,v_1),\;e_2=(v_1,v_2),\;e_3=(v_1,w),\;\\
e_4=(w,v_2),\;e_5=(v_2,u_2),\;e_6=(w,u_2).\end{array}$$
There are no additional edges in $B(\md)$ except the ones just specified.
The heights of $e_1,...,e_6$ are, in this order, $1,0,1,1,1,0$. (All this is
determined by the braid \rref{ebraid2} - see Figure 11.) 

\vspace{3mm}

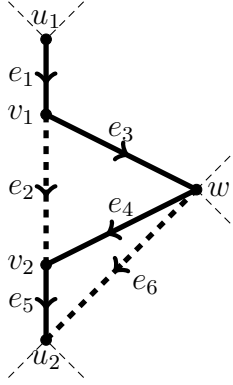
\begin{figure}
\label{fedges}
\begin{tikzpicture}
\draw[dashed][line width=2pt] (0,1) -- (0,-1);
\draw[line width=2pt] (0,1) -- (2,0);
\draw[line width=2pt] (2,0) -- (0,-1);
\draw[line width=2pt] (0,-1) -- (0,-2);
\draw[dashed][line width=2pt] (0,-2) -- (2,0);
\draw[line width=2pt] (0,1) -- (0,2);

\draw[densely dashed] (-.5,2.5) -- (0,2);
\draw[densely dashed] (.5,2.5) -- (0,2);
\draw[densely dashed] (-.5,-2.5) -- (0,-2);
\draw[densely dashed] (.5,-2.5) -- (0,-2);
\draw[densely dashed] (2,0) -- (2.5,.5);
\draw[densely dashed] (2,0) -- (2.5,-.5);

\draw[->][line width=2pt] (0,1.5) -- (0,1.4);
\draw[->][line width=2pt] (0,0) -- (0,-.1);
\draw[->][line width=2pt] (0,-1.5) -- (0,-1.6);
\draw[->][line width=2pt] (1,-1) -- (.9,-1.1);
\draw[->][line width=2pt] (.9,-.55) -- (.8,-.6);
\draw[->][line width=2pt] (1,.5) -- (1.1,.45);

\draw (0,1.5) node[anchor=east] {$e_1$};
\draw (0,0) node[anchor=east] {$e_2$};
\draw (1,.5) node[anchor=south] {$e_3$};
\draw (1,-.5) node[anchor=south] {$e_4$};
\draw (0,-1.5) node[anchor=east] {$e_5$};
\draw (1,-1) node[anchor=north west] {$e_6$};

\draw (0,2) node[anchor=south] {$u_1$};
\draw (0,1) node[anchor=east] {$v_1$};
\draw (0,-1) node[anchor=east] {$v_2$};
\draw (2,0) node[anchor=west] {$w$};
\draw(0,-2) node[anchor=north] {$u_2$};

\filldraw[black] 
(0,2) circle (2pt)
(0,1) circle (2pt)
(0,-1) circle (2pt)
(2,0) circle (2pt)
(0,-2) circle (2pt);

\end{tikzpicture}
\caption{The neighborhood of the black graph associated to 
the $b^{-1}a^{-1}baba^{-1}$ unbraid: 
$h(e_1)=1, h(e_2)=0, h(e_3)=1, h(e_4)=1, h(e_5)=1, h(e_6)=0$}
\end{figure}

\vspace{3mm}
Let $K$ be a field with elements satisfying the assumptions \rref{cass}
for the link projection $\md$.
We will now describe a field $\tilde{K}$ which can be used
to calculate BOS cohomology for the projection $\md^\prime$.
Concretely, consider the field $\tilde{K}$ obtained from $K$
by attaching new variables $A,B,C,D,E,F$, each individually algebraically
independent of $K$, corresponding to the edges
$e_1,...,e_6$ (in that order and orientation), and the same variables for the
corresponding white edges. On forming
$\alpha(T,T^\prime)$, $\beta(T,T^\prime)$ in 
$C(\md^\prime)$, we adopt the same convention as in the case of the R2 move,
i.e. for a black or white circuit $c$ occuring in a graph (a tree
plus or minus one edge) containing the edges $e_i$, $i\in I$, and
not the edges $e_j$, $j\in \{1,...,6\}\smin I$, denote by $\approx_I$
the equivalence relation on $\{u_1,u_2,v_1,v_2,w\}$ of being contained
in the same connected component of the graph formed by the edges 
$e_i$, $i\in I$. Then
take the appropriate element of $K$ assigned to the circuit in the
graph obtained by contracting each equivalence class of $\approx_I$ to a point,
and omitting the edges $e_j$, $j\in \{1,...,6\}\smin I$, and then multiply
by those of the variables $A,B,C,D,E,F$ which occur in $c$
or which $c$, or their inverses, according to orientation. 
Our only additional assumption on the field $\tilde{K}$ at this point is that the differential 
\rref{epsip} make sense for $\md^\prime$ in the sense that the denominators
be non-zero.

\vspace{3mm}
Now if we attempted to use an analogue of the Second spectral sequence directly,
i.e. filter $C(\md^\prime)$ by the height contributions
from the edges $e_i$, $i=1,...,6$, the associated graded complex will be
non-trivial in $5$ different degrees (the associated graded piece in filtration
degree $0$ turns out to be trivial). This situation seems too complicated
to analyze directly by a diagram analogous to \rref{euqu}.

\vspace{3mm}
This is where the First spectral sequence becomes relevant: Let $F^pC(\md^\prime)$ be
the filtration by height contributions of all the edges except $e_1,...,e_6$.
More precisely, then, again, $F^pC(\md^\prime)$ is spanned by all spanning trees $T$
of $B(\md^\prime)$ such that
\beg{r3e1}{
\sum_{e\in EB(\md):\phi(e)\in ET} h(e) +\sum_{e\in EB(\md):\phi(e)\notin ET}
(1-h(e))\geq p.
}
Call the sum of the height contributions of the edges $e_i$ the {\em $e$-degree}:
$$e(T):=\sum_{i:e_i\in ET} h(e_i)+\sum_{i:e_i\notin ET}(1-h(e_i)).$$
For a subset $I\subset\{1,...,6\}$, denote
by $K_I$ the set of all spanning trees $T$ of $B(\md^\prime)$ such that 
$e_i\in ET$ if and only if $i\in I$. 

Then the following table specifies the $e$-degrees of the sets $K_I$:
$$\begin{array}{ll}
\text{$e$-degree} & \text{$I$}\\
1 & \{126\},\{236\},\{246\},\{256\}
\\
2& \{12\},\{23\},\{24\},\{25\},\{1246\},\{1256\},\{1236\}
\\
3& \{125\},\{123\},\{124\},\{235\},\{245\},\{156\},\{356\},\{146\},\{346\}
\\
4&\{1235\},\{1245\},\{1346\},\{1356\},\{14\},\{15\},\{34\},\{35\}
\\
5&\{134\},\{135\},\{145\},\{345\}
\\
6&\{1345\}.
\end{array}$$
Then $d_0$ in the spectral sequence associated with the filtration $F$ is
given by all contributions of the differential which raise $e$-degree by $2$.
This can actually be determined by cancellation, since sets $K_I$, $K_J$
are bijective when the equivalence relations $\sim_I$, $\sim_J$ on $\{u_1,u_2,w\}$
coincide
where $\sim_I$ is the equivalence relation of being in the same connected component
of the forest with edges $\{e_i\;|\;i\in I\}$. To simplify notation, let us 
write $I$ instead of $K_I$.

\begin{lemma}
\label{lr3cancel}
Assume the field $\tilde{K}$ either satisfies the assumption \rref{cass} for the link projection 
$\mathcal{D}^\prime$, or is the fraction field of the quotient of the polynomial ring
$K[A,B,C,D,E,F]$ by the ideal
\beg{er3ideal}{(AD-CE).
}
Then 

\begin{enumerate}

\item
\label{er3cancel1}
$d_0$ maps isomorphically the free $\tilde{K}$-module on 
$$\{12\},\{23\},\{24\},\{25\}$$
of $e$-degree $2$ to the free $\tilde{K}$-module on 
$$\{14\},\{15\},\{34\},\{35\}$$
of $e$-degree $4$. 

\item With respect to $d_0$, the free $\tilde{K}$-modules on
$$\{1246\},\{1256\},\{1236\},$$
in $e$-degree $2$,
$$\{1235\},\{1245\},\{1346\},\{1356\}$$
in $e$-degree $4$, and 
$$\{1345\}$$
of $e$-degree $6$ form a short exact sequence. 

\item
\label{er3cancel2}
$d_0$ maps the free $\tilde{K}$-module on
$$\{126\},\{236\},\{246\},\{256\}$$
of $e$-degree $1$ isomorphically to the free $\tilde{K}$-module on
$$\{156\},\{356\},\{146\},\{346\}.$$

\item
$d_0$ maps the free $\tilde{K}$-module
on 
$$\{123\},\{124\},\{235\},\{245\}$$
in degree $3$
isomorphically to the free $\tilde{K}$-module on
$$\{135\}, \{134\}, \{345\}, \{145\}.$$
\end{enumerate}
\end{lemma}

\Proof
As it turns out, all of the claims with the exception of \rref{er3cancel1} and 
\rref{er3cancel2} follow from the shape of the graph whose vertices are
the generators and edges are pairs for which $d_0$ has a non-zero coefficient.
In the case of \rref{er3cancel1} and \rref{er3cancel2}, the coefficients
themselves must be considered in order to prove
that one gets an isomorphism. This can be done by direct computation,
which turns out to be the same in the case of either \rref{er3cancel1} or \rref{er3cancel2}.
Consider the graph $\Gamma$ depicted in Figure 12.

\begin{figure}
\label{fedges}
\begin{tikzpicture}
\draw[dashed][line width=2pt] (0,1) -- (0,-1);
\draw[line width=2pt] (0,1) -- (2,0);
\draw[line width=2pt] (2,0) -- (0,-1);
\draw[line width=2pt] (0,-1) -- (0,-2);
\draw[dotted] (0,-2) -- (2,0);
\draw[dotted] (0,2) --(2,0);
\draw[line width=2pt] (0,1) -- (0,2);

\draw (1,1) node[anchor=east] {$a$};
\draw (1,0) node[anchor=east] {$b$};
\draw (1,-1) node[anchor=east] {$c$};

\draw (0,2) node[anchor=south] {$u_1$};
\draw (0,1) node[anchor=east] {$v_1$};
\draw (0,-1) node[anchor=east] {$v_2$};
\draw (2,0) node[anchor=west] {$w$};
\draw(0,-2) node[anchor=north] {$u_2$};

\filldraw[black] 
(0,2) circle (2pt)
(0,1) circle (2pt)
(0,-1) circle (2pt)
(2,0) circle (2pt)
(0,-2) circle (2pt);

\end{tikzpicture}
\caption{The graph $\Gamma$}
\end{figure}
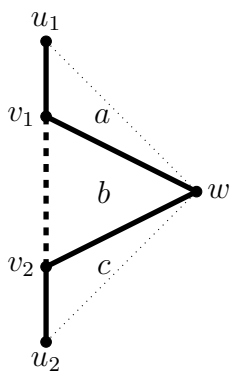

In the case of \rref{er3cancel1}, the thin dotted lines are paths in a tree $T$ to which
we are applying $d_0$. In the case of
\rref{er3cancel2}, the thin dotted line between $u_1$ and $w$ is a path in $T$, the thin dotted
line between $u_2$ and $w$ is the edge $e_6$. 

Let us consider the case of \rref{er3cancel1} (since the other case is the same). The differential
$d_0(T)$ with $ET\smallsetminus \{e_1,...,e_6\}$ fixed
can be expressed as BOS differential in $\Gamma$. Then submatrix of the differential
we are interested in with a given $ET\smallsetminus \{e_1,...,e_6\}$
has columns $\{12\}$, $\{23\}$, $\{24\}$, $\{25\}$ and rows
$\{14\}$, $\{15\}$, $\{34\}$, $\{35\}$. The non-zero coefficients are at the entries
$(\{14\},\{12\})$, $(\{14\},\{24\})$, $(\{15\},\{12\})$,  $(\{15\},\{25\})$,
$(\{34\},\{23\})$, $(\{34\},\{24\})$, $(\{35\},\{23\})$, 
$(\{35\},\{25\})$. (The entire matrix is a block sum of such matrices.)
It follows that the differential has two terms, which, when expressed as BOS differential
coefficients for the graph $\Gamma$ with vertex base point $w$ and faces as labeled, is
$$\begin{array}{l}
\left(\fracd{1}{1+ab}+\fracd{1}{1+v_2}\right)\left(\fracd{1}{1+b^{-1}}+\fracd{1}{1+v_1}\right)\cdot\\[4ex]
\left(\fracd{1}{1+bc}+\fracd{1}{1+v_2}\right)\left(\fracd{1}{1+a^{-1}b^{-1}c^{-1}}+\fracd{1}{1+v_1}\right)-\\[4ex]
\left(\fracd{1}{1+abc}+\fracd{1}{1+v_2}\right)\left(\fracd{1}{1+a^{-1}b^{-1}}+\fracd{1}{1+v_1}\right)\cdot
\\[4ex]
\left(\fracd{1}{1+b}+\fracd{1}{1+v_2}\right)\left(\fracd{1}{1+b^{-1}c^{-1}}+\fracd{1}{1+y}\right).
\end{array}$$
Using, for example, SAGE, which is capable of computing in fields of rational functions over $\F_2$,
one easily checks that this is non-zero (note that the relation \rref{er3ideal} translates to $v_1=v_2$).
\qed

From Lemma \ref{lr3cancel}, we see that after $d_0$, the only summand left is 
\beg{espec3}{\{125\}
}
in $e$-degree $3$, which is the result we wanted, since adding the
edges $e_1,e_2,e_5$ gives a bijection
$$\kappa:K(\md)\r K_{\{125\}}.$$
Additionally, since the $E_1$-term is entirely in $e$-degree $3$, we have
$$2(p+q)=p+3,$$
which means that the $E_1$-term is on a line of slope $-1/2$, the only
possible differential is $d_2$, and the spectral sequence collapses to 
$E_3$. Thus, we have reduced
our statement to showing that ``a modification of the differential
on $C(\md)$'' has isomorphic cohomology.
Additionally, the modification can be computed using the method of
\rref{eapp12}.

\vspace{3mm}

\begin{lemma}
\label{lr3d2}
Let $\tilde{K}$ be as in the assumptions of Lemma \ref{lr3cancel}. Then, identifying a spanning tree of 
$B(\md^\prime)$ which contains the edges $e_1,e_2,e_5$  with the spanning tree of $B(\md)$ obtained
by contracting the edges $e_1,e_2,e_5$ to a point, the differential
$d_2$ of our spectral sequence is related to the differential in $C(\md)$ by a conjugation followed
by an automorphism of fields.
\end{lemma}

\vspace{3mm}

Note that this implies Proposition \ref{pbraid} immediately.

\vspace{3mm}
\noindent
{\bf Proof of Lemma \ref{lr3d2}:} 
The general recipe for calculating $d_2$ is to take a $d_0$-cocycle $t$,
and add to the component of $dt$ in the same $e$-degree the following term:
apply to $t$ the component $\nu$ of the differential which raises $e$-degree
by $1$. Then, by assumption, $\nu(t)\in \operatorname{Im} d_0$, say, $\nu(t)=d_0(s)$. Then we have
\beg{eformd2}{d_2(t)=\nu(s).}
Now $u_i$ and $w$ are in the same component
of $T\smallsetminus \{e_2\}$ for precisely one value $i=1,2$. Assume $i=1$ (the case $i=2$ is treated analogously).
Then, a $d_0$-cocycle is represented by a tree $T\in K_{\{125\}}$ plus a
a multiple of a tree in $K_{\{146\}}$; one sees that for an appropriate
choice of the multiple, this is indeed a $d_0$-cocycle; on the other hand, by
Lemma \ref{lr3cancel}, any $d_0$-cohomology class may be represented as a linear
combination of elements of this form.

Applying $\nu$ gives three summands
\beg{e15}{\sum_\alpha\left(\frac{1}{1+r_\alpha BDF}+\frac{1}{1+b_\alpha ABE}\right)T_{15}^{\alpha},
}
\beg{e1235}{\sum_{\beta}\left(\frac{1}{1+b_\beta AC}+\frac{1}{1+q_\beta C^{-1}DF}\right)T_{1235}^{\beta},
}
\beg{e1245}{\sum_\beta\left(\frac{1}{1+b_\beta ABD^{-1}}+\frac{1}{1+q_\beta C^{-1}DF}\right)T_{1245}^{\beta}.}
Here $T_I$ (with a superscript as needed) indicates a tree in $K_I$, $b_\beta\in \tilde{K}$ is the element associated to the
black circuit arising from $T$ in the graph obtained by identifying all the
edges $e_i$ to a single point and $q_\beta\in \tilde{K}$ is an element assigned according
to the above convention to a white circuit corresponding to the connected
component containing the vertex $w$ and $r_\alpha\in \tilde{K}$ is the element assigned
to a white cycle containing $v_2$ (according to the right hand rule).

As remarked above, in general, we do not know that $T$ is a $d_0$-cocycle. In fact, it may support a $d_0$
which lands in $K_{145}$. This may in turn be canceled by adding an element of $K_{146}$
with the appropriate coefficient (see Figure 13, where $d_0$'s are denoted by 
a single arrow and $\nu$'s by double arrows). 

\begin{figure}
\begin{tikzpicture}[line width = 1.5pt]
\draw (0,2) -- (0,1);
\draw[dashed] (0,1) -- (0,0);
\draw (0,0) -- (0,-1);

\filldraw
(0,2) circle (2pt)
(0,1) circle (2pt)
(0,0) circle (2pt)
(0,-1) circle (2pt)
(1,.5) circle (2pt);

\draw[->] (0,-1.75) --(0,-3.25);

\draw (0,-5) -- (0,-4);
\draw (0,-6) -- (1,-5.5);
\draw (0,-7) -- (0,-6);

\filldraw
(0,-4) circle (2pt)
(0,-5) circle (2pt)
(0,-6) circle (2pt)
(0,-7) circle (2pt)
(1,-5.5) circle (2pt);

\draw (0,-11) -- (0,-10);
\draw (0,-12) -- (1,-11.5);
\draw[dashed] (0,-13) -- (1,-11.5);

\draw[->] (0,-9.25) --(0,-7.75);

\draw[->] (3.5,-10.25) -- (3.5,-11.75);
\draw[->] (3.5,-5.75) -- (3.5,-4.25);

\filldraw
(0,-10) circle (2pt)
(0,-11) circle (2pt)
(0,-12) circle (2pt)
(0,-13) circle (2pt)
(1,-11.5) circle (2pt);

\draw[double][->] (1.5,1.5) -- (2.5,2.5);
\draw[double][->] (1.5,-.5) -- (2.5,-1.5);

\draw[double][->] (1.5,-12.5) -- (2.5,-13.5);

\draw (3.5,4.5) -- (3.5,3.5);
\draw[dashed] (3.5,3.5) -- (3.5,2.5);
\draw (3.5,2.5) -- (3.5,1.5);
\draw (3.5,3.5) -- (4.5,3);

\filldraw
(3.5,4.5) circle (2pt)
(3.5,3.5) circle (2pt)
(3.5,2.5) circle (2pt)
(3.5,1.5) circle (2pt)
(4.5,3) circle (2pt);

\draw (3.5,-.5) -- (3.5,-1.5);
\draw[dashed] (3.5,-1.5) -- (3.5,-2.5);
\draw (3.5,-2.5) -- (3.5,-3.5);
\draw (3.5,-2.5) -- (4.5,-2);

\filldraw
(3.5,-.5) circle (2pt)
(3.5,-1.5) circle (2pt)
(3.5,-2.5) circle (2pt)
(3.5,-3.5) circle (2pt)
(4.5,-2) circle (2pt);

\draw (3.5,-12.5) -- (3.5,-13.5);
\draw[dashed] (3.5,-15.5) -- (4.5,-14);
\draw (3.5,-13.5) -- (4.5,-14);
\draw (3.5,-14.5) -- (4.5,-14);

\filldraw
(3.5,-6.5) circle (2pt)
(3.5,-7.5) circle (2pt)
(3.5,-8.5) circle (2pt)
(3.5,-9.5) circle (2pt)
(4.5,-8) circle (2pt);

\draw (3.5,-6.5) -- (3.5,-7.5);
\draw[dashed] (3.5,-9.5) -- (4.5,-8);
\draw[dashed] (3.5,-7.5) -- (3.5,-8.5);
\draw (3.5,-8.5) -- (4.5,-8);

\filldraw
(3.5,-12.5) circle (2pt)
(3.5,-13.5) circle (2pt)
(3.5,-14.5) circle (2pt)
(3.5,-15.5) circle (2pt)
(4.5,-14) circle (2pt);

\draw[<-] (5,2.5) -- (6,1.5);
\draw[<-] (5,-1.5) -- (6,-.5);

\draw (6.75,2) -- (6.75,1);
\draw[dashed] (6.75,1) -- (6.75,0);
\draw (6.75,0) -- (6.75,-1);
\draw[dashed] (6.75,-1) -- (7.75,.5);

\filldraw
(6.75,2) circle (2pt)
(6.75,1) circle (2pt)
(6.75,0) circle (2pt)
(6.75,-1) circle (2pt)
(7.75,.5) circle (2pt);

\draw[double][->] (8.5,.5) -- (10,.5);

\draw (10.75,2) -- (10.75,1);
\draw[dashed] (10.75,1) -- (10.75,0);
\draw (10.75,0) -- (10.75,-1);

\filldraw 
(10.75,2) circle (2pt)
(10.75,1) circle (2pt)
(10.75,0) circle (2pt)
(10.75,-1) circle (2pt)
(11.75,.5) circle (2pt);

\end{tikzpicture}

\caption{}
\end{figure}

We are really only interested in the top row of Figure 13. 
Let $T_{1256}^{\beta}$ be obtained from $T_{1235}^{\beta}$ by replacing $e_3$ with $e_6$
(or equivalently from $T_{1245}^{\beta}$ by replacing $e_4$ with $e_6$).
Then
\beg{e1256}{\begin{array}{l}
\displaystyle
d_0T_{1256}^{\beta}=\left(\frac{1}{1+C^{-1}BEF^{-1}}+\frac{1}{1+q_\beta 
C^{-1}DF}\right)T_{1235}^{\beta}+\\[4ex]
\displaystyle
\left(\frac{1}{1+DEF^{-1}}+\frac{1}{1+q_\beta C^{-1}DF}\right)T_{1245}^{\beta}.
\end{array}
}
(The coefficients on the right hand side may not match those of $\nu(T)$, but
this can be remedied using the contribution of the parallelogram on the bottom left of Figure 13.)
Now there exist trees $T_{125}^{\gamma}$ and  elements $c_\gamma\in \tilde{K}$ associated to black
circuits under the above convention such that 
\beg{e1256t}{\nu(T_{1256}^{\beta})=\sum_\gamma\left(
\frac{1}{1+b_\beta c_\gamma ABEF^{-1}}+\frac{1}{1+q_\gamma C^{-1}DF}\right)T_{125}^{\gamma}
+\text{other terms}.
}

\vspace{3mm}
(Here and below, ``other terms'' means linear combinations of elements in other sets $K_I$
which do not contribute to the differential.)
Next, let $T_{12}^{\alpha}$ resp. $T_{25}^{\alpha}$ be obtained from $T_{15}^{\alpha}$ by replacing $e_5$
with $e_2$ resp. $e_1$ with $e_2$. Then 
\beg{e12}{d_0(T_{12}^{\alpha})=
\left(\frac{1}{1+b_\alpha ABE}+\frac{1}{1+BDE^{-1}}\right)T_{15}^{\alpha}+\text{other terms},
}
\beg{e25}{d_0(T_{25}^{\alpha})=
\left(\frac{1}{1+b_\alpha ABE}+\frac{1}{1+BCA^{-1}}\right)T_{15}^{\alpha}+\text{other terms}.}
The other terms do occur and will make it necessary to add counterterms of
$e$-degree $2$ in $K_J$ for where the cardinality of $J$ is $2$.
However, it is easy to see that $\{12\}$, $\{25\}$ are the only
terms which, after applying $\nu$, can produce a non-zero multiple of a tree
in $K_{\{125\}}$. 

More specifically, there exist additional elements
$p_\delta\in \tilde{K}$ assigned to a white  tree $T_{125}^{\prime\delta}$ such that
\beg{e12t5}{\nu(T_{12}^{\alpha})=\sum_\delta
\left(\frac{1}{1+b_\alpha ABE}+\frac{1}{1+r_\alpha p_\delta EF}\right)T^{\prime\delta}_{125}+\text{other terms},
}
\beg{e25t1}{\nu(T_{25}^{\alpha})=\sum_\delta
\left(\frac{1}{1+b_\alpha ABE}+\frac{1}{1+r_\alpha p_\delta AC^{-1}DF}\right)T^{\prime\delta}_{125}+
\text{other terms}.
}
Again, the other terms are trees not in $K_{\{125\}}$ which cancel by Lemma \ref{lr3cancel}.
Thus, if we impose the additional relation
\beg{eeq2}{AD=CE,
}
the contributions of the trees $T_{12}$ and $T_{25}$ to $d_2$ will be
equal, and it suffices to consider one of them. (See Figure 14.)

\begin{figure}
\begin{tikzpicture}[line width = 1.5pt]

\draw (0,2) -- (0,1);
\draw[dashed]  (0,1) -- (0,0);
\draw (0,0) -- (0,-1);

\filldraw[black]
(0,2) circle (2pt)
(0,1) circle (2pt)
(0,0) circle (2pt)
(0,-1) circle (2pt)
(1,.5) circle (2pt);

\draw[double][->](1.75,.5) -- (3.25,.5);

\draw (4,2) -- (4,1);
\draw (4,0) -- (4,-1);

\filldraw[black]
(4,2) circle (2pt)
(4,1) circle (2pt)
(4,0) circle (2pt)
(4,-1) circle (2pt)
(5,.5) circle (2pt);

\draw[->] (6.5,2.5) -- (5.5,1.5);

\draw (7.5,4.5) -- (7.5,3.5);
\draw[dashed] (7.5,3.5) -- (7.5,2.5);

\filldraw[black]
(7.5,4.5) circle (2pt)
(7.5,3.5) circle (2pt)
(7.5,2.5) circle (2pt)
(7.5,1.5) circle (2pt)
(8.5,3) circle (2pt);

\draw[->] (6.5,-1.5) -- (5.5,-.5);

\draw[dashed] (7.5,-1.5) -- (7.5,-2.5);
\draw (7.5,-2.5) -- (7.5,-3.5);

\filldraw[black]
(7.5,-.5) circle (2pt)
(7.5,-1.5) circle (2pt)
(7.5,-2.5) circle (2pt)
(7.5,-3.5) circle (2pt)
(8.5,-2) circle (2pt);

\draw[double][->] (9,2.5) -- (10,1.5);
\draw[double][->] (9,-1.5) -- (10,-.5);

\draw (10.75,2) -- (10.75,1);
\draw[dashed] (10.75,1) -- (10.75,0);
\draw (10.75,0) -- (10.75,-1);

\filldraw
(10.75,2) circle (2pt)
(10.75,1) circle (2pt)
(10.75,0) circle (2pt)
(10.75,-1) circle (2pt)
(11.75,.5) circle (2pt);

\end{tikzpicture}
\caption{}
\end{figure}

Thus, to summarize, assuming \rref{eeq2}, $d_2 T$ is obtained
by adding to the component of $dT$ in $K_{\{125\}}$ the term
\beg{eaddon}{
\begin{array}{l}
\displaystyle
\sum\left(\frac{1}{1+b_\alpha ABE}+\frac{1}{1+r_\alpha p_\delta EF}\right)\left(\frac{1}{1+b_\alpha ABE}+
\frac{1}{1+BDE^{-1}}
\right)^{-1}\cdot\\[4ex]
\displaystyle\left(\frac{1}{1+r_\alpha BDF}+\frac{1}{1+b_\alpha ABE}\right)T_{125}^{\prime\delta}+\\[4ex]
\displaystyle\sum 
\left(\frac{1}{1+b_\beta c_\gamma ABEF^{-1}}+\frac{1}{q_\beta C^{-1}DF}\right)
\left(\frac{1}{1+C^{-1}BEF^{-1}}+\frac{1}{q_\beta C^{-1}DF}\right)^{-1}\cdot\\[4ex]
\displaystyle
\left(\frac{1}{1+b_\beta AC}+\frac{1}{1+q_\beta C^{-1}DF}\right)T_{125}^{\gamma}.
\end{array}
}
Therefore, 
each of the summands of \rref{eaddon} is of the same form 
as the summand $i\eta^{-1}j$ in \rref{edeform} for appropriate graphs, making 
R2 moves on the edges $e_2$, $e_5$ and $e_4$, $e_6$. The graphs concerned
consist of the edges depicted in Figure 14 resp. in the top two rows of Figure 13, and 
the edges of $B(\md^\prime)\smallsetminus \{e_1,\dots,e_6\}$.
Note that the argument used in Section \ref{sr2} to prove that the differential \rref{edeform}
obtained by an R2 move has the same BOS cohomology as the original differential
depends only on the maps $i,\eta, j$ being coefficients of the corresponding parts of
a BOS differential, which is replicated in the present case. 
The case of the edges $e_4$ and
$e_6$ is related to the case discussed in Section \ref{sr2} by black-white duality. (Note that
in the case of $e_4$, $e_6$, the edge $e_5$ does not move by $d_2$, and hence
can be ignored; we may simply contract $e_5$ from the graph, multiplying the adjacent
face variables by $E$ ofr $E^{-1}$, depending on orientation.) 
Lemma \ref{lxyzt} and the discussion following it therefore imply
the part of the statement of Lemma \ref{lr3d2} which assumes the relation \rref{eeq2}.
(Note that the two summands result in two operations each consisting of
a conjugation and a change of variables; one operation, however, only concerns 
black circuits and hence face variables, and the other only concerns white circuits
and hence vertex variables. Therefore they commute.)

It remains to discuss the removal of the relation \rref{eeq2}. To this end, however, looking
at the second graph from the left in Figure 14, we see that it has three connected
components, and therefore the equality of differentials we proved 
involves two independent white circuit variables. Multiplying one of them by a new
variable algebraically independent from the rest destroys the relation \rref{eeq2}, and
renders the variables $A,B,C,D,E,F$ jointly algebraically independent over $K$.
\qed

\vspace{3mm}
\noindent
{\bf Proof of Theorem \ref{t1p}:} We have shown that, subject to
the condition \rref{lass}, the numbers \rref{erank}
are invariant under the three Reidemeister moves. However,
there are still some minor details left
to finish proving the Theorem: When $L$ is a knot, we are claiming that
\rref{erank} is independent of orientation. This is simply because the number
$n_-$ does not depend on orientation in this case. 

When, on the other hand, $L$ is a link which has a projection violating the
condition \rref{lass}, we claim that \rref{erank} is $0$ (and therefore
also a link invariant even in this case). To this end, it suffices to prove
that \rref{erank} is $0$ for a projection $\md$ which will become disconnected
by a single reversed R2 move. 

In such a case, however, there always exist two vertices $v_1,v_2$ of
$B(\md)$ connected by two edges $e,f$ of heights $0$, $1$ respectively, 
where $v_1$ and $v_2$ are in different connected components of
$B(\md)\smin\{e,f\}$. In this case, however, $C(\md)$ has the form
of the chain complex $\overline{U}$ of Lemma \ref{ltheta} where 
$U_{2}^{\prime}$ resp. $U_2$ is generated by spanning trees containing
the edge $e$ (resp. $f$). Thus, $C(\md)$ is acyclic.
\qed

\vspace{5mm}

\section{A few computations}
\label{scomp}

The purpose of this Section is to give a few examples of computations of
BOS cohomology (which we will 
denote by $H_{BOS}$ here), to give a basic idea of its behavior. BOS appears
to be a sparse invariant, close in flavor to twisted $\widehat{HF}$. There
is, (see \cite{os}), a single-graded
spectral sequence $E_r$ (i.e. graded like the Bockstein specral sequence)
whose $E_3$-term is $H_{BOS}$,
converging to twisted $\widehat{HF}$. This spectral sequence is sparse 
in the sense that
the only possible non-zero differentials are of the form $d_{4k+2}$ (\cite{os}).
One may ask if this spectral sequence always collapses to $E_3$. This is unknown at
present, but even if this is the case, BOS cohomology contains additional
information due to the grading, which is intrinsically different from
gradings on (twisted) $\widehat{HF}$, which are given by $spin^c$-structures.

\vspace{3mm}
Let $\sigma(L)$
denote the signature of a link $L$.  We call a link $L$ {\em BOS thin} if
$$\rank(H^{i}_{BOS}(L))=\left\{\begin{array}{ll} \det(L)&\text{when $i=\sigma(L)/2$}\\
0 & \text{else.}
\end{array}\right.$$
The simplest computation of BOS cohomology is the following

\vspace{3mm}

\begin{proposition}
\label{palt}
Every alternating link $L$ is BOS thin.
\end{proposition}

\vspace{3mm}
\Proof
Choose an alternating projection $\mathcal{D}(L)$ and a checkerboard 
coloring so that all black edges have height $1$. Then clearly $\mc(\mathcal{D}(L))$
is concentrated in a single dimension $i$, and the number of spanning trees
is equal to $\det(L)$ by Kirkhoff's theorem. To calculate $i$, let $b$ be the number of black vertices. Then
$$i=(b-1-n_{-})/2.$$
This number is equal to $\sigma(L)/2$ by \cite{tra} (see also \cite{daslow} for
more results on that subject). \qed

\vspace{3mm}
In \cite{os2}, Ozsv\'{a}th and Szab\'{o} define a class $\mathcal{A}$
of {\em quasi-alternating links} which is the smallest class containing the
unknot such that if a link $L$ has resolutions $L_0$ and $L_1$ 
at a particular crossing in some projection such that $L_0, L_1\in\mathcal{A}$, and
\beg{edet2}{\det(L_0)+\det(L_1)=\det(L),
}
then $L\in\mathcal{A}$. For our purposes, it is useful to extend this notion further.
Let a class $\widetilde{\mathcal{A}}$ of {\em weakly quasi-alternating
links} be defined the same way as $\mathcal{A}$, except that we replace, in
the above definition, ``the
unknot'' by ``the unknot and all split links''. It is proved in \cite{os2} that all non-split 
alternating links are quasi-alternating. It follows immediately that all 
alternating links are weakly quasi-alternating.

The statement of Proposition \ref{palt} extends to weakly quasi-alternating
links. We state this separately, since the argument involves a much
deeper step due to Manolescu and Ozsv\'{a}th \cite{mo}.  Let us start with the following

\vspace{3mm}

\begin{lemma}
\label{lskein}
We have a long exact sequence
\beg{elskein}{\begin{array}{l}
\dots\r H^{i+n_-(\md_1)/2-1/2}(L_1)\r
H^{i+n_-(\md)/2}(L)\r\\
\r H^{i+n_-(\md_0)/2}(L_0)\r H^{i+n_-(\md_1)/2+1/2}(L_1)\r\dots
\end{array}
}
\end{lemma}

\vspace{3mm}
\Proof
The long exact sequence \rref{elskein} is the spectral sequence concentrated
in filtration degrees $0$ and $1$ associated with the decreasing filtration
on $\mc(\md)$ where $F^\epsilon\mc(\md)$ is the set of linear combinations
of all those spanning trees $T$ where the edge $e$ (whether it belongs to $T$ or
not) contributes $\geq \epsilon$ to the height.
\qed

\vspace{3mm}
\noindent
{\bf Comment:} The reason for the $n_-/2$ summands in the degrees in \rref{elskein} is
that the grading of the spectral (=exact) sequence is specified by
height alone, and cannot include the number of negative crossings which
we subtracted in the grading of BOS cohomology, and thus must add back
on in the long exact sequence. (To see this, note that the $L_0$ and $L_1$
resolutions cannot both preserve the orientation of the link, and hence
the numbers of negative crossings can change unpredictably.) Note also
that as a result of this, the terms of the long exact sequence \rref{elskein}
are not (oriented or unoriented) link invariants. 

\vspace{3mm}

\begin{proposition}
\label{pqalt}
Every weakly quasi-alternating link is BOS thin.
\end{proposition}

\Proof
Completely analogous to the proof of Theorem 1 of Manolescu-Ozsv\'{a}th \cite{mo}:
If the $0$-resolutions $L_0$ and $L_1$ of a link $L$ are BOS-thin, have non-zero determinant
and 
$$\det(L)=\det(L_0)+\det(L_1),$$
then $L$ is BOS thin by Lemma \ref{lskein} and Lemma 3 of \cite{mo}.
\qed

\vspace{3mm}
\noindent
{\bf Comment:} The subtle fact that we can use `weakly quasi-alternating'
instead of `quasi-alternating' in Proposition \ref{pqalt} is actually interesting,
since Theorem 1 of \cite{mo} does not hold for weakly quasi-alternating links.
This is, therefore, a first application of BOS cohomology: it can be used
to prove that a link is not weakly quasi-alternating (see Corollary \ref{ct37}
below), which is a stronger statement than proving that a link is
not quasi-alternating.

\vspace{3mm}

\begin{lemma}
\label{lqalt}
If, for a link $L$, all the values of $i$ for which $H^{i}_{BOS}(L)\neq 0$
differ by integral multiples of $2$, then 
$$\sum_i \rank(H^{i}_{BOS}(L))=\rank(\underline{\widehat{HF}}(L))=\det(L).$$
(Here $(\underline{\widehat{HF}}$ denotes the twisted Heegaard-Floer homology with
coefficients in the Novikov ring, as considered in \cite{os}.)
\end{lemma}

\vspace{3mm}
\Proof
Because of sparsity of the Baldwin-Ozsv\'{a}th-Szab\'{o} spectral sequence
(the only differentials being $d_{4k+2}$), the spectral sequence collapses
under the assumption. Further, in general,
the Euler characteristic of $\mc(\mathcal{D}(L))$
is
equal to the Euler characteristic of the based Khovanov complex at $q=-1$, which
is $\det(L)$. Under the given assumption, the Euler characteristic is
equal to the total rank.
\qed

\vspace{3mm}
In view of these observations, it is natural to ask if there exist a knot whose
BOS-cohomology is not concentrated in a single degree. Such knots do indeed
exist. 

\vspace{3mm}

\begin{proposition}
\label{pnontriv}
The torus knot $T(3,7)$ has non-trivial BOS cohomology in at least two degrees
whose difference is not an even integer. 
\end{proposition}

\Proof
The branched double cover $\Sigma(T(3,7))$
of $T(3,7)$
is the Brieskorn homology sphere with multiplicities $2,3,7$. Since this
is a homology $3$-sphere, twistings are homologically trivial
and hence $\underline{\widehat{HF}}(\Sigma(T(3,7))$
and $\widehat{HF}(\Sigma(T(3,7)),\Z/2)$ have the same rank. The latter group
is computed in \cite{os3}, p. 209: one has
$$HF^+(\Sigma(T(3,7)))=\Z[U,U^{-1}]/U\Z[U]\oplus \Z,$$
and hence 
$$\rank(\widehat{HF}(\Sigma(T(3,7)),\Z/2)=3$$
by the universal coefficient theorem. Hence, $H_{BOS}(\Sigma(T(3,7))$
cannot be concentrated in degrees which differ by even integers by Lemma \ref{lqalt}.

\qed

\vspace{3mm}
\begin{corollary}
\label{ct37}
The knot $T(3,7)$ is not weakly quasi-alternating. 
\end{corollary}

\Proof
Apply proposition \ref{pqalt}.
\qed

\vspace{3mm}
\noindent
{\bf Concluding remarks:} Before finding out about $T(3,7)$, 
the authors considered the three
smallest non-alternating knots, $8_{19}$, $8_{20}$ and $8_{21}$ in Rolfsen's
table (see Figure 15)  in search of a non-trivial example.

\begin{figure}

\begin{tikzpicture}[line width=2pt]

\draw (-4,0) .. controls (-5,-1) and (-6.5,0) .. (-6,.8) .. controls (-5.5,1.5) .. (-5,2.5);

\draw (-4.75,2.9) .. controls (-4.25,3.5) and (-4,3.25) .. (-3.75,2.9) .. controls (-3.5,2.5) .. (-3.25,2);

\draw (-3,1.6) .. controls (-2,0) and ( -2,-1) .. (-3.5,0) .. controls (-5,1) .. (-4.5,2);

\draw (-4.25,2.25) .. controls (-4,2.5) .. (-3.75,2.6);

\draw (-3.4,2.6) .. controls (-3,2.5) and (-3,2)  .. (-3.25,1.575) .. controls (-3.5,1.25) and (-3.75,1) .. (-4.5,1.4);

\draw (-4.9,1.65) ..controls (-5,1.7) .. (-5.1,1.8);

\draw (-5.4,2.1) .. controls (-5.75,2.5) and (-5.25,2.9) .. (-5,2.79) .. controls (-4.8,2.7) and (-4.5,2.5) .. (-4,1.4);

\draw (-3.75,1) .. controls (-3.5,.6) and (-3.6,.4) .. (-3.67,.35);

\end{tikzpicture}

\begin{tikzpicture}[line width=2pt]

\draw (0,1.25*1.25*.1) .. controls (1.25*1.25*-.25,1.25*1.25*-.5) and (1.25*-1,1.25*-.5) .. (1.25*-1,0) .. controls (1.25*-1,1.25*.5) and (1.25*-.75,1.25*1.25) .. (1.25*-.25,1.25*1.25);

\draw (1.25*.25,1.25*1.2) .. controls (1.25*1,1.25*1) and (1.25*1.25,1.25*.75) .. (1.25*1.1,1.25*.5);

\draw (1.25*1,1.25*.25) .. controls (1.25*.5,1.25*-1) and (1.25*.25,1.25*-1) .. (1.25*-.5,1.25*-1.05) .. controls (1.25*-1,1.25*-1) and (1.25*-2.1,1.25*-1) .. (1.25*-2,1.25*.5) .. controls (1.25*-2,1.25*1) and (1.25*-1.5,1.25*1.3) .. (1.25*-1.45,1.25*1.25);

\draw (1.25*-1.15,1.25*1.23) .. controls (1.25*-1,1.25*1.2) .. (1.25*-.8,1.25*1.1);

\draw (1.25*-.5,1.25*.95) .. controls (1.25*-.25,1.25*.8) and (1.25*.25,1.25*.6) .. (1.25*1,1.25*.4) .. controls (1.25*1.35,1.25*.3) and (1.25*1.3,0) .. (1.25*1.15,1.25*-.05) .. controls (1.25*1.15,1.25*-.05) .. (1.25*1, 1.25*-.075);

\draw (1.25*.75,0) .. controls (1.25*.5,.25) and (0,1.25*.2) .. (1.25*-.75,1.25*.3);

\draw (1.25*-1.15,1.25*.45) .. controls (1.25*-1.3,1.25*.5) and (1.25*-1.35,1.25*1) .. (1.25*-1.3,1.25*1.25) .. controls (1.25*-1.2,1.25*1.8) and (1.25*-.25,1.25*1.8) .. (1.25*.04,1.25*1.3)  .. controls (1.25*.2,1.25*1) and (1.25*.2,1.25*.9) .. (1.25*.2,1.25*.8);

\draw (1.25*.15,1.25*.55) .. controls (1.25*.125,1.25*.5) and (1.25*.1,1.25*.4) .. (1.25*.075,1.25*.35);

\end{tikzpicture}

\begin{tikzpicture}[line width=2pt]

\draw (-.2,.25)  .. controls (.5,-.75) and (.5,-1.5) .. (.3,-2) .. controls (0,-2.75) and (-.5,-2.5) .. (-.75,-2.25);

\draw (-1,-2) .. controls (-1.25,-1) and (-1,-.5) .. (-.5,.25) .. controls (0,1)  and (.5,1) .. (1,0);

\draw (1.25,-.5) .. controls (1.5,-1.5) and (1.5,-2.5) .. (1,-3) .. controls (0,-4) and (-1,-3.5) .. (-1.5,-3) .. controls (-2.5,-2) and (-2,-1) .. (-1.5,-.575) .. controls (-1.35,-.45) .. (-1.15,-.4);

\draw (-.625,-.35) .. controls (-.375,-.325) .. (-.15,-.31);

\draw (.35,-.275) .. controls (.75,-.25) and (1,-.21) .. (1.25,-.225) .. controls (1.5,-.25) and (1.75,-.4) .. (1.45,-.625);

\draw (1.1,-.8) .. controls (1,-.9) and (.8,-.8) .. (.6,-1);

\draw (.25,-1.275) .. controls (.15,-1.3) and (0,-1.75) .. (-.75,-2.1) .. controls (-1.5,-2.5) and (-1.75,-1.75) .. (-1.825,-1.35);

\draw (-2,-.75) .. controls (-2.25,0) and (-1.75,.5) .. (-1.25,.75) .. controls (-1,.9) and (-.75,.8) .. (-.6,.625);

\end{tikzpicture}
\caption{The knots $8_{19}$, $8_{20}$ and $8_{21}$ in Rolfsen's table.}
\end{figure}
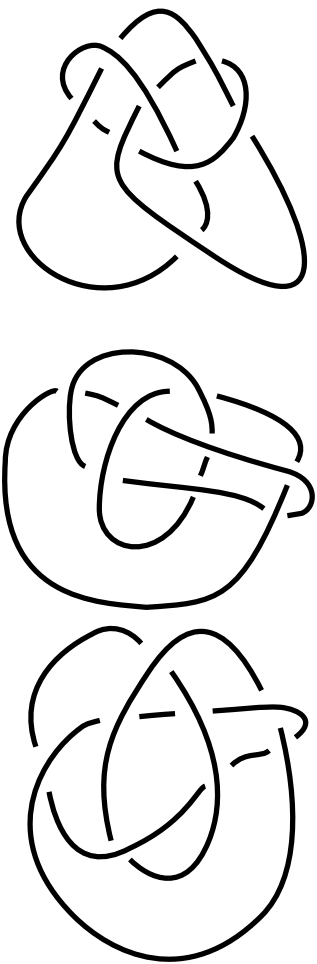

From that point of view, however, these knots prove
to be disappointing. While one can compute the BOS cohomology of the knots $8_{20}$, $8_{21}$
directly from the definition, they are also known to be quasi-alternating 
(Manolescu \cite{man}), so they are BOS thin. 

The knot $8_{19}$ is not quasi-alternating (since it is not Khovanov thin - see \cite{dbar}),
but nevertheless is BOS thin. This is because 
$8_{19}$ has a pretzel projection, with black graph depicted in Figure 16. An algorithm for
computing Heegaard-Floer homology of branched double covers of pretzel knots is given
in \cite{os3}, and it is known that the $\widehat{HF}$ of $8_{19}$ is $(\Z/2)^3$
(while the determinant of $8_{19}$ is $3$).
On the other hand, the BOS complex of the black graph in Figure 16 is non-trivial
only in two adjacent degrees. Therefore, there is no room for differentials or extensions
in the BOS twisted spectral sequence, and we conclude that $8_{19}$ is also BOS thin. 
It remains an open problem whether or not $8_{19}$ is weakly quasi-alternating. 

We also computed the BOS cohomology
of $8_{19}$ from the definition. The full computation with algebraically independent variables
is beyond the range of Maple and Mathematica on computers with 16GB memory. We introduced,
however, an algebraic dependency between the variables (making them all powers of a single
transcendental variable), and got cohomology concentrated in a single degree equal to half
the signature. By Proposition \ref{generic}, we were therefore able to conclude that
$8_{19}$ is BOS thin. Later we discovered that MAGMA (and SAGE) are somewhat more efficient
at computing in fields of rational functions over $\F_2$, and were able to also compute a basis
of the BOS cohomology of $8_{19}$ in algebraically independent variables. The result
is several megabytes long. 

While the present paper was under review, I.Kriz and E. Elmanto \cite{ke}
did compute compute explicitly examples of knots and links with non-trivial BOS cohomologies
(including a precise computation of the BOS cohomology of $T(3,7)$).
Computable examples were very hard to obtain, and the general computation
of BOS cohomology with algebraically independent variables appears unworkable for
larger knots at present. It is worth remarking, however, that there is much current interest in manifolds
whose $\widehat{HF}$ has rank equal to the number of elements of their first integral
homology (so called $L$-spaces). Branched double covers of BOS thin knots are examples
of $L$-spaces, and Proposition \ref{generic} does give a computationally efficient way
of detecting such examples.

\begin{figure}
\begin{tikzpicture}[line width = 1.5pt]
\draw (0,2) -- (2*0.877583,1);
\draw (2*0.877583,1) -- (2*0.877583,-1);
\draw (2*0.877583,-1) -- (0,-2);
\draw (0,-2) -- (-2*0.877583,-1);
\draw (-2*0.877583,-1) -- (-2*0.877583,1);
\draw (-2*0.877583,1) -- (0,2);
\draw[dashed] (0,2) -- (0,0);
\draw[dashed] (0,0) -- (0,-2);

\filldraw
(0,0) circle (2pt)
(0,2) circle (2pt)
(2*0.877583,1) circle (2pt)
(2*0.877583,-1) circle (2pt)
(0,-2) circle (2pt)
(-2*0.877583,-1) circle (2pt)
 (-2*0.877583,1) circle (2pt);

\end{tikzpicture}
\caption{The black graph of a pretzel projection of $8_{19}$}
\end{figure}
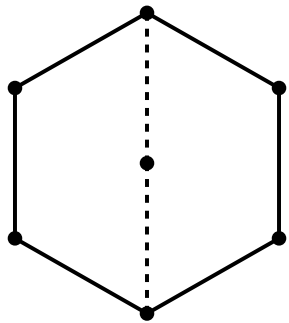


\vspace{10mm}
\begin{thebibliography}{99}

\bibitem{baldwin} J.Baldwin: On the spectral sequence from Khovanov homology
to Heegaard Floer homology, {\em Int. Math. Res. Not.} (2011) 3426-3470

\bibitem{cancel} J.Baldwin: On transverse invariants in knot Floer and 
HOMFLY-PT homologies, www.math.princeton.edu/$\sim$baldwinj/trans4.pdf

\bibitem{bl} J.Baldwin, A.S.Levine: A combinatorial spanningtree model for knot
Floer homology, {\em Adv. Math.} 231 (2012) 1886-1939

\bibitem{os} J.Baldwin, P.Ozsv\'{a}th, Z.Szab\'{o}: Heegaard Floer homology of
double-covers, Kauffman states, and Novikov rings, to appear



\bibitem{dbar} Dror Bar-Natan: On Khovanov's categorification of the Jones
polynomial, {\em Alg. Geom. Topology} 2 (2002) 337-370

\bibitem{daslow} O.Dasbach, A.Lowrance: Turaev genus, knot signature, and the
knot homology concordance invariants, {\em Proc. Amer. Math. Soc.} 139 (2011) 2631-2645

\bibitem{ke} E.Elmanto, I.Kriz: Some non-trivial examples of the Baldwin-Ozsv\'{a}th-Szab\'{o}
twisted spectral sequence and Heegaard-Floer homology of branched double covers, to appear

\bibitem{jaeg} T.C.Jaeger: A Remark on Roberts' totally twisted Khovanov homology, arXiv: 1109.1805


\bibitem{kauf} L.H.Kauffman: {\em On knots}, Ann. Math. Studies 115, Princeton
Univ. Press, 1987

\bibitem{khovanov} M.Khovanov: A categorification of the Jones polynomial,
{\em Duke Math. J.} 101 (2000) 359-426

\bibitem{man} C.Manolescu: An unoriented skein exact triangle for
knot Floer homology, {\em Math. Res. Lett.} 14 (2007) 839-852

\bibitem{mo} C.Manolescu, P.Ozsv\'{a}th: On Khovanov and knot Floer homologies
of quasi-alternating links, {\em Proc. G\"{o}kova Geometry/Topology conference},
2007, pp. 60-81, G\"{o}kova Geometry/Topology Conference, G\"{o}kova, 2008

\bibitem{mos} C.Manolescu, P. Ozsv\'{a}th, P.Sarkar: A combinatorial 
description of knot Floer homology,
{\em Ann. of Math.} (2) 169 (2009), no. 2, 633-660




\bibitem{os2} P.Ozsv\'{a}th, Z.Szab\'{o}: On the Heegaard Floer homology 
of branched double-covers, {\em Adv. Math.}  194  (2005),  no. 1, 1-33

\bibitem{ost} P.Ozsv\'{a}th, Z.Szab\'{o}: Holomorphic disks and topological
invariants for closed three-manifolds, {\em Ann. of Math.} 159 (2004) 1027-1158

\bibitem{ost2} P.Ozsv\'{a}th, Z.Szab\'{o}: Holomorphic disks and knot invariants.  
{\em Adv. Math.} 186  (2004),  no. 1, 58-116

\bibitem{ost3} P.Ozsv\'{a}th, Z.Szab\'{o}: Holomorphic disks, link invariants 
and the multi-variable Alexander polynomial.  {\em Algebr. Geom. Topol.}  
8  (2008),  no. 2, 615-692

\bibitem{os3} P.Ozsv\'{a}th, Z.Szab\'{o}: On the Floer homology of plumbed
three-manifolds, {\em Geom. Top.} 7 (2003) 185-224

\bibitem{roberts} L.Roberts: Totally twisted Khovanov homology, arxiv: 1109.1805

\bibitem{tra} P.Traczyk: A combinatorial formula for the signature of
alternating diagrams, {\em Fund. Math.} 184 (2004) 311-316

\end{thebibliography}
\end{document}